\documentclass[12pt,oneside, italian]{article}

\usepackage[english]{babel}

\usepackage{cancel}

\usepackage{cite}

\usepackage{latexsym}

\usepackage{amssymb,amsthm,amsmath}

\usepackage{longtable}
\usepackage{graphicx,color}

\newtheorem{theorem}{Theorem}[section]

\newtheorem{proposition}[theorem]{Proposition}

\newtheorem{lemma}[theorem]{Lemma}

\newtheorem{corollary}[theorem]{Corollary}

\newtheorem{remark}[theorem]{Remark}

{\theoremstyle{definition}

\newtheorem*{definition*}{Definition}

}

\newtheorem*{proposition*}{Proposition}

\newtheorem*{corollary*}{Corollary}

\newtheorem*{lemma*}{Lemma}

\newtheorem*{remark*}{Remark}

\newcommand{\cX}{\mathcal X}
\newcommand{\cL}{\mathcal L}

\newcommand{\fq}{\mathbb {F}_q}

\newcommand{\aut}{\mathrm{Aut}}
\newcommand{\fqq}{\mathbb{F}_{q^2}}

\newcommand{\fqnq}{\mathbb{F}_{q^{2n}}}

\title{AG codes and AG quantum codes from the GGS curve}
\author{D. Bartoli, M. Montanucci, G. Zini}

\date{}

\begin{document}

\maketitle 

\begin{abstract}
In this paper, algebraic-geometric (AG) codes associated with the GGS maximal curve are investigated. The Weierstrass semigroup at all $\mathbb F_{q^2}$-rational points of the curve is determined; the Feng-Rao designed minimum distance is computed for infinite families of such codes, as well as the automorphism group. As a result, some linear codes with better relative parameters with respect to one-point Hermitian codes are discovered.
Classes of quantum and convolutional codes are provided relying on the constructed AG codes.
\end{abstract}


\footnote{
This research was partially supported by Ministry for Education, University
and Research of Italy (MIUR) (Project PRIN 2012 ``Geometrie di Galois e
strutture di incidenza'' - Prot. N. 2012XZE22K$_-$005)
 and by the Italian National Group for Algebraic and Geometric Structures
and their Applications (GNSAGA - INdAM).

URL: Daniele Bartoli (daniele.bartoli@dmi.unipg.it), Maria Montanucci (maria.montanucci@unibas.it), Giovanni Zini (gzini@math.unifi.it).
}

\section{Introduction}

In \cite{Goppa1,Goppa2} Goppa used algebraic curves to construct linear error correcting codes, the so called algebraic geometric codes (AG codes).
The construction of an AG code with alphabet a finite field $\mathbb F_q$ requires that the underlying curve is $\mathbb F_q$-rational and involves two $\mathbb F_q$-rational divisors $D$ and $G$ on the curve.

In general, to construct a ``good" AG code over $\mathbb F_q$ we need a curve $\mathcal X$ with low genus $g$ with respect to its number of $\mathbb F_q$-rational points. In fact, from the Goppa bounds on the parameters of the code it follows that the relative Singleton defect is upper bounded by the ratio $g/N$, where $N$ can be as large as the number of $\mathbb F_q$-rational points of $\mathcal X$ not in the support of $G$.
Maximal curves over $\mathbb F_q$ attain the Hasse-Weil upper bound for the number of $\mathbb F_q$-rational points with respect to their genus and for this reason they have been used in a number of works. Examples of such curves are the Hermitian curve, the GK curve \cite{GK2009}, the GGS curve \cite{GGS}, the Suzuki curve \cite{DL1976}, the Klein quartic when $\sqrt{q}\equiv6\pmod7$ \cite{MEAGHER2008}, together with their quotient curves. Maximal curves often have large automorphism groups which in many cases can be inherited by the code: this can bring good performances in encoding \cite{Joyner2005} and decoding \cite{HLS1995}.

Good bounds on the parameters of one-point codes, that is AG codes arising from divisors $G$ of type $nP$ for a point $P$ of the curve, have been obtained by investigating the Weierstrass semigroup at $P$. These results have been later generalized to codes and semigroups at two or more points; see e.g. \cite{MATTHEWS2001,HOMMA1996,HK2001,CT2005,CK2009,LC2006,Kim1994}.

AG codes from the Hermitian curve have been widely investigated; see \cite{HK2006,HK2005,Tiersma1987,DK2011,HK2006_2,YK1991,Stichtenoth1988} and the references therein. Other constructions based on the  Suzuki curve and the curve with equation $y^q + y = x^{q^r+1}$ can be found in \cite{Matthews2004} and \cite{ST2014}.
More recently, AG Codes from the GK curve have been constructed in \cite{FG2010,CT2016,BMZ}.

In the present work we investigate one-point AG codes from the $\mathbb F_{q^{2n}}$-maximal GGS curve, $n\geq5$ odd.
The GGS curve has more short orbits under its automorphism group than other maximal curves, see \cite{GOS}, and hence more possible structures for the Weierstrass semigroups at one point.
On the one hand this makes the investigation more complicated; on the other hand it gives more chances of finding one-point AG codes with good parameters.
One achievement of this work is the determination of the Weierstrass semigroup at any $\mathbb F_{q^2}$-rational point.

We show that the one-point codes at the infinite point $P_\infty$ inherit a large automorphism group from the GGS curve; for many of such codes, the full automorphism group is obtained.
Moreover, for $q=2$, we compute explicitly the Feng-Rao designed minimum distance, which improves the Goppa designed minimum distance.
As an application, we provide families of codes with $q=2$ whose relative Singleton defect goes to zero as $n$ goes to infinity.
We were not able to produce analogous results for an $\mathbb F_{q^2}$-rational affine point $P_0$, because of the more complicated structure of the Weierstrass semigroup.
In a comparison between one-point codes from $P_\infty$ and one-point codes from $P_0$, it turns out that the best codes come sometimes from $P_\infty$, other times from $P_0$; we give evidence of this fact with tables for the case $q=2$, $n=5$.

Note that in general, many of our codes are better than the comparable one-point Hermitian codes on the same alphabet.
In fact, let $C_{1}$ be a code from a one-point divisor $G_{1}$ on the $\mathbb F_{q^{2n}}$-maximal GGS curve with genus $g_1$, with alphabet $\mathbb F_{q^{2n}}$, length $N_{2}$, designed dimension $k_{1}^*=\deg G_{1}-g_{1}+1$, and designed minimum distance $d_{1}^*=\deg G_{1}-(2g_{1}-2)$. In the same way, let $C_{2}$ be a code from a one-point divisor $G_{2}$ on the $\mathbb F_{q^{2n}}$-maximal Hermitian curve with genus $g_2$, with the same alphabet $\mathbb F_{q^{2n}}$ and length $N_2=N_{1}$ as $C_{1}$, designed dimension $k_{2}^*=\deg G_{2}-g_{2}+1$, and designed minimum distance $d_{2}^*=\deg G_{2}-(2g_{2}-2)$.
In order to compare $C_1$ and $C_2$, we can choose $G_1$ and $G_2$ such that $k_1^*=k_2^*$. Then the difference $d_1^*-d_2^*$, like the difference $\delta_2^*-\delta_1^*$ between the designed Singleton defects, is equal to $g_2-g_1=\frac{1}{2}(q^{2n}-q^{n+2}+q^3-q^2)\gg0$.

Finally, we apply our results on AG codes to construct families of quantum codes and convolutional codes.

\section{Preliminaries}\label{Sec:Preliminaries}

\subsection{Curves and codes}\label{Sec:Preliminaries_Curves}

Let $\mathcal{X}$ be a projective, geometrically irreducible, nonsingular algebraic curve of genus $g$ defined over the finite field $\mathbb{F}_q$ of size $q$.
The symbols $\mathcal{X}(\mathbb{F}_q)$ and $\mathbb{F}_q(\mathcal{X})$ denote the set of $\mathbb{F}_q$-rational points and the field of $\mathbb{F}_q$-rational functions, respectively.
A divisor $D$ on $\mathcal{X}$ is a formal sum $n_1P_1+\cdots+n_rP_r$, where $P_i \in \mathcal{X}(\mathbb{F}_q)$, $n_i \in \mathbb{Z}$, $P_i\neq P_j$ if $i\neq j$.
The divisor $D$ is $\mathbb F_q$-rational if it coincides with its image $n_1P_1^q+\cdots+n_rP_r^q$ under the Frobenius map over $\mathbb F_q$.
For a function $f \in \mathbb{F}_q(\mathcal{X})$, $div(f)$ and $(f)_{\infty}$ indicate the divisor of $f$ and its pole divisor.
Also, the Weierstrass semigroup at $P$ will be indicated by $H(P)$.
The Riemann-Roch space associated with an $\mathbb F_q$-rational divisor $D$ is
$$\mathcal{L}(D) := \{ f \in \mathcal{X}(\mathbb{F}_q)\setminus\{0\} \ : \ div(f)+D \geq 0\}\cup\{0\}$$
and its dimension over $\mathbb{F}_q$ is denoted by  $\ell(D)$.

Let $P_1,\ldots,P_N\in \mathcal{X}(\mathbb{F}_q)$ be pairwise distinct  points and consider the divisor $D=P_1+\cdots+P_N$ and another $\mathbb F_q$-rational divisor $G$ whose support is disjoint from the support of $D$. The AG code $C(D,G)$ is the image of the linear map $\eta :  \mathcal{L}(G) \to \mathbb{F}_q^N$ given by $\eta(f) = (f(P_1),f(P_2) ,\ldots,f(P_N))$. The code has length $N$ and if $N>\deg(G)$  then $\eta$ is an embedding and the dimension $k$ of $C(D,G)$ is equal to $\ell(G)$. The minimum distance $d$ satisfies $d\geq d^*=N-\deg(G)$, where $d^*$ is called the designed minimum distance of $C(D,G)$; if in addition $\deg(G)>2g-2$, then by the Riemann-Roch Theorem $k=\deg(G)-g+1$; see \cite[Th. 2.65]{HLP}. The dual code $C^{\bot} (D,G)$ is an $AG$ code with dimension $k^{\bot}=N-k$ and minimum distance $d^{\bot}\geq \deg{G}-2g+2$. If $G=\alpha P$, $\alpha \in \mathbb{N}$, $P \in \mathcal{X}(\mathbb{F}_q)$, the AG codes ${C} (D,G)$ and ${C}^{\bot} (D,G)$ are referred to as one-point AG codes. Let $H(P)$ be the Weierstrass semigroup associated with $P$, that is 
$$H(P) := \{n \in \mathbb{N}_0 \ | \ \exists f \in \mathbb{F}_q(\mathcal{X}), (f)_{\infty}=nP\}= \{\rho_1=0<\rho_2<\rho_3<\cdots\}.$$

Denote by $f_{\ell}\in \mathbb{F}_q(\mathcal{X})$, $\ell\geq 1$, a rational function such that $(f_{\ell})_{\infty}=\rho_{\ell}P$. For $\ell \geq0$, define the \emph{Feng-Rao function} 
$$\nu_\ell := | \{(i,j) \in \mathbb{N}_0^2 \ : \ \rho_i+\rho_j = \rho_{\ell+1}\}|.$$ 
Consider ${C}_{\ell}(P)= {C}^{\bot}(P_1+P_2+\cdots+P_N,\rho_{\ell}P)$, $P,P_1,\ldots,P_N$ pairwise distint points in $\mathcal{X}(\mathbb{F}_q)$. The number 
$$d_{ORD} ({C}_{\ell}(P)) := \min\{\nu_{m} \ : \ m \geq \ell\}$$
is a lower bound for the minimum distance $d({C}_{\ell}(P))$ of the code ${C}_{\ell}(P)$, called the \emph{order bound} or the \emph{Feng-Rao designed minimum distance} of ${C}_{\ell}(P)$; see \cite[Theorem 4.13]{HLP}. Also, by \cite[Theorem 5.24]{HLP}, $d_{ORD} ({C}_{\ell}(P))\geq \ell+1-g$ and equality holds if $\ell \geq 2c-g-1$, where $c=\max \{m \in \mathbb{Z} \ : \ m-1 \notin H(P)\}.$ 

A numerical semigroup is called telescopic if it is generated by a sequence $(a_1,\ldots,a_k)$ such that 
\begin{itemize}
\item $\gcd(a_1, \ldots , a_k)=1$;
\item for each $i=2,\ldots,k$, $a_i/d_i \in \langle a_1/d_{i-1},\ldots, a_{i-1}/d_{i-1}\rangle$, where $d_i=\gcd(a_1,\ldots,a_i)$;
\end{itemize}
see \cite{KP}.
The semigroup $H(P)$ is called symmetric if $2g-1\notin H(P)$. The property of being symmetric for $H(P)$ gives rise to useful simplifications of the computation of $d_{ORD}(C_\ell(P))$, when $\rho_\ell >2g$. The following result is due to Campillo and Farr\'an; see \cite[Theorem 4.6]{CF}.

\begin{proposition} \label{campillo} Let $\cX$ be an algebraic curve of genus $g$ and let $P \in \cX(\mathbb{F}_q)$. If $H(P)$ is a symmetric Weierstrass semigroup then one has $$d_{ORD}(C_\ell(P))=\nu_{\ell},$$ for all $\rho_{\ell+1}=2g-1+e$ with $e \in H(P) \setminus \{0\}$. 
\end{proposition}

\subsection{The automorphism group of an AG code $C(D,G)$}

In the following we use the same notation as in \cite{GK2008,JK2006}.
Let $\mathcal{M}_{N,q}\leq{\rm GL}(N,q)$ be the subgroup of matrices having exactly one non-zero element in each row and column.
For $\gamma\in{\rm Aut}(\fq)$ and $M=(m_{i,j})_{i,j}\in{\rm GL}(N,q)$, let $M^\gamma$ be the matrix $(\gamma(m_{i,j}))_{i,j}$.
Let $\mathcal{W}_{N,q}$ be the semidirect product $\mathcal M_{N,q}\rtimes{\rm Aut}(\fq)$ with multiplication $M_1\gamma_1\cdot M_2\gamma_2:= M_1M_2^\gamma\cdot\gamma_1\gamma_2$.
The \emph{automorphism group} ${\rm Aut}({C}(D,G))$ of ${C}(D,G)$ is the subgroup of $\mathcal{W}_{N,q}$ preserving ${C}(D,G)$, that is,
$$ M\gamma(x_1,\ldots,x_N):=((x_1,\ldots,x_N)\cdot M)^\gamma \in {C}(D,G) \;\;\textrm{for any}\;\; (x_1,\ldots,x_N)\in {C}(D,G). $$
Let ${\rm Aut}_{\fq}(\cX)$ denote the $\fq$-automorphism group of $\cX$. Also, let
$$ {\rm Aut}_{\fq,D,G}(\cX)=\{ \sigma\in{\rm Aut}_{\fq}(\cX)\,\mid\, \sigma(D)=D,\,\sigma(G)\approx_D G \}, $$
where $G'\approx_D G$ if and only if there exists $u\in\fq(\cX)$ such that $G'-G=(u)$ and $u(P_i)=1$ for $i=1,\ldots,N$, and
$$ {\rm Aut}_{\fq,D,G}^+(\cX):=\{ \sigma\in{\rm Aut}_{\fq}(\cX)\,\mid\, \sigma(D)=D,\,\sigma(|G|)=|G| \}, $$
where $|G|=\{G+(f)\mid f\in\overline{\mathbb F}_q(\cX)\}$ is the linear series associated with $G$.
Note that ${\rm Aut}_{\fq,D,G}(\cX)\subseteq {\rm Aut}_{\fq,D,G}^+(\cX)$. 

\begin{remark}\label{Coincidono}
Suppose that ${\rm supp}(D)\cup{\rm supp}(G)=\cX(\fq)$ and each point in ${\rm supp}(G)$ has the same weight in $G$. Then
$$ {\rm Aut}_{\fq,D,G}(\cX) = {\rm Aut}_{\fq,D,G}^+(\cX) = \{\sigma\in{\rm Aut}_{\fq}(\cX)\,\mid\,\sigma({\rm supp}(G))={\rm supp}(G) \}. $$
\end{remark}

In \cite{GK2008} the following result was proved.
\begin{theorem}{\rm(\!\!\cite[Theorem 3.4]{GK2008})}\label{Aut}
Suppose that the following conditions hold:
\begin{itemize}
\item $G$ is effective;
\item $\ell(G-P)=\ell(G)-1$ and $\ell(G-P-Q)=\ell(G)-2$ for any $P,Q\in\cX$;
\item $\cX$ has a plane model $\Pi(\cX)$ with coordinate functions $x,y\in\cL(G)$;
\item $\cX$ is defined over $\mathbb F_p$;
\item the support of $D$ is preserved by the Frobenius morphism $(x,y)\mapsto(x^p,y^p)$;
\item $N>\deg(G)\cdot\deg(\Pi(\cX))$.
\end{itemize}
Then
$$ {\rm Aut}({C}(D,G))\cong ({\rm Aut}_{\fq,D,G}^+(\cX)\rtimes{\rm Aut}(\fq))\rtimes \mathbb{F}_q^*. $$
\end{theorem}

If any non-trivial element of ${\rm Aut}_{\fq}(\cX)$ fixes at most $N-1$ $\fq$-rational points of $\cX$  then ${\rm Aut}({C}(D,G))$ contains a subgroup isomorphic to $ ({\rm Aut}_{\fq,D,G}(\cX)\rtimes{\rm Aut}(\fq))\rtimes \mathbb{F}_q^*$; see \cite[Proposition 2.3]{BMZ}.

\subsection{The GGS curve}
Let $q$ be a prime power and consider an odd integer $n$. 
The GGS curve $GGS(q,n)$ is defined by the equations
\begin{equation}\label{GGS_equation}
\left\{
\begin{array}{l}
X^q + X = Y^{q+1}\\
Y^{q^2}-Y= Z^m\\
\end{array}
\right.
,
\end{equation}
where $m= (q^n+1)/(q+1)$; see \cite{GGS}.
The genus of $GGS(q,n)$ is $\frac{1}{2}(q-1)(q^{n+1}+q^n-q^2)$, and $GGS(q,n)$ is $\mathbb F_{q^{2n}}$-maximal.

Let $P_0=(0,0,0)$, $P_{(a,b,c)}=(a,b,c)$, and let $P_{\infty}$ be the unique ideal point of $GGS(q,n)$.
Note that $GGS(q,n)$ is singular, being $P_\infty$ its unique singular point. Yet, there is only one place of $GGS(q,n)$ centered at $P_\infty$; therefore, we can actually construct AG codes from $GGS(q,n)$ as described in Section \ref{Sec:Preliminaries_Curves} (see \cite[Appendix B]{Sti} and \cite[Chapter 8]{HKT} for an introduction to the concept of place of a curve).
The divisors of the functions $x,y,z$ satisfying $x^q + x = y^{q+1}$, $y^{q^2}-y= z^m$ are
$$
(x)=m(q+1)P_0-m(q+1)P_{\infty},
$$
$$
(y)=m\sum_{\alpha^q+\alpha=0} P_{(\alpha,0,0)}-mqP_{\infty},
$$
$$
(z)=\sum_{\scriptsize\begin{array}{l} \alpha^q+\alpha=\beta\\ \beta \in \mathbb{F}_{q^2}\\ \end{array}} P_{(\alpha,\beta,0)}-q^3P_{\infty}.
$$
Throughout the paper we indicate by $\overline D$ and $\tilde D$ the divisors 
\begin{equation}\label{Dbarra}
\overline D=\sum_{P\in GGS(q,n)(\mathbb{F}_{q^{2n}}),\ P\ne P_{\infty}}P,\qquad \tilde D=\sum_{P\in GGS(q,n)(\mathbb{F}_{q^{2n}}),\ P\ne P_{0}}P.
\end{equation}

\subsection{Structure of the paper}

The paper is organized as follows. In Section \ref{Sec:dord2} the value of $d_{ORD}(C_{\ell}(P_\infty))$ for $q=2$ and  $n\geq 5$ is obtained, where $C_\ell(P_\infty)=C^{\perp}(\overline{D},\ell P_{\infty})$; this is applied in Section \ref{Sec:Application1} to two families of codes with $q=2$ whose relative Singleton defect goes to zero as $n$ goes to infinity.
In Section \ref{Sec:P_0} we determine the Weierstrass semigroup at $P_0$, and hence at any $\mathbb F_{q^2}$-rational affine point of $GGS(q,n)$.
The tables in Section \ref{Sec:tabelle} describe the parameters of $C_\ell{P_\infty}$ and $C_\ell(P_0)$ in the particular case $q^{2n}=2^{10}$.
Sections \ref{Sec:Application2} and \ref{Sec:Application3} provide families of quantum codes and convolutional codes constructed from $C_\ell(P_\infty)$ and $C_\ell(P_0)$.
Finally, we compute in Section \ref{Sec:Auto} the automorphism group of the AG code $C(\overline{D},\ell P_{\infty})$ for $q^n+1 \leq \ell \leq q^{n+2}-q^3$.



\section{The computation of $d_{ORD}(C_{\ell}(P_\infty))$ for $q=2$}\label{Sec:d_ORD}\label{Sec:dord2}

In this section we deal with the codes $C_{\ell}(P_\infty)={C}^{\bot}(\overline D,\rho_{\ell}P_\infty)$, where $\overline D$ is as in \eqref{Dbarra}. Our purpose is to exhibit the exact value of $d_{ORD}(C_{\ell}(P_\infty))$ for the case $q=2$. First of all we determine the values of $\nu_\ell$ (Subsection \ref{Sec:d_ORD_sub1}); in Subsections  \ref{Sec:d_ORD_sub2}, \ref{Sec:d_ORD_sub3}, \ref{Sec:d_ORD_sub4} we compute $d_{ORD}(C_{\ell}(P_\infty))$.

\subsection{The Feng-Rao function $\nu_\ell$ for $q=2$}\label{Sec:d_ORD_sub1}

Assume that $q=2$ and $n \geq 5$ is odd. Let $m=\frac{2^n+1}{3}$. Then, from \cite[Corollary 3.5]{GOS},
$$H(P_\infty)= \bigg\{i(2^n+1)+2j \frac{2^n+1}{3}+8k \mid i \in \{0,1\}, \ j \in \{0,1,2,3\}, \ k \geq 0 \bigg\}.$$

\begin{remark} \label{scrittura} Let $\rho_\ell=i(2^n+1)+2j \frac{2^n+1}{3}+8k \in H(P_\infty)$. Then $\rho_\ell$ is uniquely determined by the triple $(i,j,k)$.
\end{remark}

\begin{proof}
Assume that $i(2^n+1)+2j \frac{2^n+1}{3}+8k=i'(2^n+1)+2j' \frac{2^n+1}{3}+8k'$. Then $i \equiv i' \pmod2$ and since $i,i' <2$ we have that $i=i'$. Thus, $2j\frac{2^n+1}{3}+8k=2j' \frac{2^n+1}{3}+8k'$. Since this implies that $j \equiv j' \pmod4$ and $j,j'<4$, we have that $j=j'$ and $k=k'$ and the claim follows.
\end{proof}

According to Remark \ref{scrittura}, the notation $(i,j,k)$ is used to indicate the non-gap at $P_\infty$ associated with the choices of the parameters $i,j,k$. In order to compute $d_{ORD}(C_{\ell}(P_\infty))$ the following definition is required. Let $\rho_\ell \in H(P_\infty)$ be fixed. Assume that $\rho_{\ell+1}=(i,j,k)$. Then,
$$\nu_\ell= \big| \{(i_r,j_r,k_r), \ r=1,2 \mid (i,j,k)=(i_1,j_1,k_1)+(i_2,j_2,k_2) \}\big|.$$

In the following lemmas we determine the value of $\nu_\ell$.

\begin{lemma} \label{i1} Let $\rho_\ell \in H(P_\infty)$ be fixed. Assume that $\rho_{\ell+1}=(1,j,k)$ for some $j=0,1,2,3$ and $k \geq 0$. Then,
$$\nu_\ell= \begin{cases} 2(j+1)(k+1), \ \ \textrm{if} \ \ k<m, \\ 2(j+1)(k+1)+2(3-j)(k-m+1), \  \ \textrm{otherwise}. \end{cases}$$
\end{lemma}

\begin{proof}
Let $i_1,i_2,j_1,j_2,k_1$, and $k_2 \in \mathbb{N}$ be such that
$$(2^n+1)+2jm+8k=(i_1+i_2)(2^n+1)+2(j_1+j_2)m+8(k_1+k_2)=3(i_1+i_2)m+2(j_1+j_2)m+8(k_1+k_2).$$
Then $i_1+i_2 \equiv 1 \pmod2$ and since $i_1+i_2 \leq 2$ we have that $i_1+i_2=1$. This implies that
$$3m+2jm+8k=3m+2(j_1+j_2)m+8(k_1+k_2),$$
and hence
\begin{equation}
\label{eq1}
jm+4k=(j_1+j_2)m+4(k_1+k_2).
\end{equation}
Assume that $j=0$. Then from (\ref{eq1}), $(j_1+j_2)m \equiv 0 \pmod 4$ and so, $j_1+j_2=4h$ for some integer $h$. Since $0 \leq j_1+j_2  \leq 6$ we have that $h=0$ or $h=1$. In the first case $k_1+k_2=k$, in the second case $k_1+k_2=k-m$. Since $k_1+k_2 \geq 0$, if $k < m$ the second case cannot occur. Thus, if $k < m$, since we have $2$ possible choices for $i_1$ and $(k+1)$ choices for $k_1$ (while $i_2$ and $k_2$ are determined according to the choices of $i_1$ and $k_1$, respectively), then $\nu_\ell=2(k+1)$.
Also, if $k \geq m$ we have that
$$\nu_\ell=2(k+1)+2\cdot|\{(j_1,j_2) : 0 \leq j_1,j_2 \leq 3, \ j_1+j_2=4\}|\cdot(k-m+1)=2(k+1)+6(k-m+1)$$
and the claim follows by direct checking. \\
Assume that $j=1$. Then from (\ref{eq1}), $(j_1+j_2)m \equiv m \pmod 4$ and so $j_1+j_2=1+4h$ for some integer $h$. Since $0 \leq j_1+j_2  \leq 6$ we have that $h=0$ or $h=1$. In the first case $k_1+k_2=k$, in the second case $k_1+k_2=k-m$. Since $k_1+k_2 \geq 0$ if $k < m$ the second case cannot occur. Thus, if $k < m$, since we have $2$ possible choices for $i_1$, $2$ possible choices for $j_1$ and $(k+1)$ choices for $k_1$, then $\nu_\ell=4(k+1)$.
Also, if $k \geq m$ we have that,
$$\nu_\ell=4(k+1)+2\cdot|\{(j_1,j_2) : 0 \leq j_1,j_2 \leq 3, \ j_1+j_2=5\}|\cdot(k-m+1)=4(k+1)+4(k-m+1),$$
and the claim follows by direct checking. \\
Assume that $j=2$. Then from (\ref{eq1}), $(j_1+j_2)m \equiv 2m \pmod4$ and so $j_1+j_2=2+4h$, for some integer $h$. Since $0 \leq j_1+j_2  \leq 6$ we have that $h=0$ or $h=1$. In the first case $k_1+k_2=k$, in the second case $k_1+k_2=k-m$. Since $k_1+k_2 \geq 0$, if $k < m$ the second case cannot occur. Thus, if $k < m$, since we have $2$ possible choices for $i_1$, $3$ possible choices for $j_1$ and $(k+1)$ choices for $k_1$, then $\nu_\ell=6(k+1)$.
Also, if $k \geq m$ we have that
$$\nu_\ell=6(k+1)+2(k-m+1)\cdot\big |\{(j_1,j_2) : 0 \leq j_1,j_2 \leq 3, \ j_1+j_2=6\}\big | =6(k+1)+2(k-m+1),$$
and the claim follows by direct checking. 

Assume that $j=3$. Then from (\ref{eq1}), $(j_1+j_2)m \equiv 3m \pmod 4$ and so $j_1+j_2=3+4h$, for some integer $h$. Since $0 \leq j_1+j_2  \leq 6$ we have that $h=0$. Since this implies that $k_1+k_2=k$, we have that $\nu_\ell=8(k+1)$.
\end{proof}

Using a similar approach we can prove the following.

\begin{lemma} \label{i0}
 Let $\rho_\ell \in H(P_\infty)$ be fixed. Assume that $\rho_{\ell+1}=(0,j,k)$ for some $j=0,1,2,3$ and $k \geq 0$. Then,
$$\nu_\ell= \begin{cases} (j+1)(k+1) + \lfloor \frac{j}{3} \rfloor (k+1), \ \ \textrm{if} \ \ k<m, \\ (j+1)(k+1) + \lfloor \frac{j}{3} \rfloor (k+1) + (5-2 \max\{0,j-2\})(k-m+1), \ \ \textrm{if} \ \ m \leq k <2m, \\ (j+1)(k+1) + \lfloor \frac{j}{3} \rfloor (k+1) + \\ + (5-2 \max\{0,j-2\})(k-m+1) + \max \{ 0,2-j\} (k-2m+1), \ \ \textrm{otherwise}. \end{cases}$$
\end{lemma}


\subsection{Computation of $d_{ORD}(C_{\ell}(P_\infty))$ for $ \rho_{\ell+1}=(1,j,k)$ and $\rho_{\ell} \leq 2g$}\label{Sec:d_ORD_sub2}

Let $\rho_\ell \in H(P_\infty)$. Assume that $\rho_{\ell+1}=(1,j,k)$ for some $j=0,1,2,3$ and $k \geq 0$.
Recall that  $C_\ell(P_\infty)$ is the dual code of the AG code $C(\overline {D},\rho_{\ell}P_{\infty})$, where $\overline D$ is as in \eqref{Dbarra}.

\begin{lemma} \label{dord1}
If $\rho_{\ell+1}=(1,0,k)$ for some $k<m$ then
$$d_{ORD}(C_\ell(P_\infty))=\begin{cases} 2, \ \ \textrm{if} \ \ k=0, \\ 3, \ \ \textrm{if} \ \ k \leq \lfloor \frac{m}{8} \rfloor, \\ 4, \ \ \textrm{if} \ \ \frac{m}{8} < k \leq \lfloor \frac{m}{4} \rfloor, \\ 5, \ \ \textrm{if} \ \ \frac{m}{4} < k \leq \lfloor \frac{3m}{8} \rfloor, \\ 6, \ \ \textrm{if} \ \ \frac{3m}{8} < k \leq \lfloor \frac{m}{2} \rfloor, \\ 8, \ \ \textrm{if} \ \ \frac{m}{2} < k \leq \lfloor \frac{3m}{4} \rfloor, \\ 8 \big(\lceil k - \frac{3m}{4}\rceil +1 \big), \ \ \textrm{if} \ \ \frac{3m}{4} \leq k \leq m-2, \\ \nu_\ell=2m, \ \ \textrm{if} \ \ k=m-1. \end{cases}$$
\end{lemma}

\begin{proof}
For $\rho_s \in H(P_\infty)$ the following system of inequalities is considered:
\begin{equation}
\label{sist}
\begin{cases} \rho_{s+1} \geq \rho_{\ell+1}, \\ \nu_s \leq \nu_\ell. \end{cases}
\end{equation}
In order to compute $d_{ORD}(C_\ell(P_\infty))$ we take the minimum value of $\nu_s$ such that System \eqref{sist} is satisfied. Also, a case-by-case analysis with respect to $a\in\{0,1\}$ is required. Assume that $\rho_{s+1}=(a,b,c)$ for some $a\in\{0,1\}$, $b\in\{0,1,2,3\}$ and $c \geq 0$. From Lemma \ref{i1}, System \eqref{sist} reads,
\begin{equation}
\label{sist1}
\begin{cases} 3am+2bm+8c \geq 3m+8k, \\ \nu_s \leq 2(k+1). \end{cases}
\end{equation}
{\bf Case 1: $a=1$ and $c<m$.} From Lemma \ref{i1}, System \eqref{sist1} reads
$$\begin{cases} 2bm+8c \geq 8k, \\ 2(b+1)(c+1) \leq 2(k+1). \end{cases}$$
\begin{itemize}
\item If $b=0$ then $c=k$ and so the unique solution is $\nu_\ell$ itself. 
\item If $b=1$ then $c \geq \lceil k-\frac{m}{4} \rceil$ and $c \leq \lfloor \frac{k-1}{2} \rfloor$. Such a $c$ exists if and only if $ \lceil k-\frac{m}{4} \rceil \leq \lfloor \frac{k-1}{2} \rfloor$.
Assume that $k$ is odd. Then $k-\lfloor \frac{m}{4} \rfloor =\lceil k-\frac{m}{4} \rceil \leq \lfloor \frac{k-1}{2} \rfloor =\frac{k-1}{2}$ if and only if $k \leq 2\lfloor \frac{m}{4} \rfloor -1$.
Similarly if $k$ is even then $c$ exists if and only if $ k-\lfloor \frac{m}{4} \rfloor \leq \frac{k-2}{2}$, that is $k \leq 2\lfloor \frac{m}{4} \rfloor -2$. For these cases the minimum is obtained taking $c=\max\{0,\lceil k-\frac{m}{4} \rceil$\} and hence $\nu_s=4(\max\{0, \lceil k-\frac{m}{4} \rceil \} +1)$.

\item If $b=2$ then $c \geq \lceil k-\frac{m}{2} \rceil$ and $c \leq \lfloor \frac{k-2}{3} \rfloor$. As before, such a $c$ exists if and only if $\lceil k-\frac{m}{2} \rceil \leq \lfloor \frac{k-2}{3} \rfloor$. 
This is equivalent to $k \leq \frac{3}{2}(\lfloor \frac{m}{2} \rfloor -1)$ if $k \equiv 0 \pmod3$, to $k \leq \frac{3}{2}\lfloor \frac{m}{2} \rfloor -2$ if $k \equiv 1 \pmod3$, to $k \leq \frac{3}{2}\lfloor \frac{m}{2} \rfloor -1$ if $k \equiv 2 \pmod3$.
For these cases the minimum is obtained taking $c=\max\{0, \lceil k-\frac{m}{2} \rceil\}$ and hence $\nu_s=6(\max\{0, \lceil k-\frac{m}{2} \rceil\} +1)$.

\item If $b=3$ then $c \geq \lceil k-\frac{3m}{4} \rceil$ and $c \leq \lfloor \frac{k-3}{4} \rfloor$.
As before, such a $c$ exists if and only if $\lceil k-\frac{3m}{4} \rceil \leq \lfloor \frac{k-3}{4} \rfloor$. 
By direct checking, this is equivalent to $k\leq m-2$.
Here the minimum is obtained taking $c=\max\{0, \lceil k-\frac{3m}{4} \rceil\}$ and hence $\nu_s=8(\max\{0, \lceil k-\frac{3m}{4} \rceil\} +1)$.
\end{itemize}
When $k > \frac{3m}{4}$ and $k \leq m-2$ the minimum value above is obtained as $\nu_s=8(\lceil k-\frac{3m}{4} \rceil +1)$. We observe that if $k=m-1$ then $\nu_\ell=2(k+1)=2m$ and $8(\max\{0, \lceil k-\frac{3m}{4} \rceil\} +1)=8(\lceil m-1-\frac{3m}{4} \rceil\} +1)>2m$. This implies that if $k=m-1$ then the minimum value is $\nu_\ell=2m$ itself.
Thus, combining the previous results we obtain
\begin{equation}
\label{min1}
\min\{\nu_s \mid a=1 \ \ \textrm{and} \ \ c<m\}=\begin{cases} 2, \ \ \textrm{if} \ \ k=0, \\ 4 \ \ \textrm{if} \ \ 1 \leq k \leq \lfloor \frac{m}{4} \rfloor, \\ 6, \ \ \textrm{if} \ \ \frac{m}{4} < k \leq \lfloor \frac{m}{2} \rfloor, \\ 8, \ \ \textrm{if} \ \ \frac{m}{2} < k \leq \lfloor \frac{3m}{4} \rfloor, \\ 8(\lceil k-\frac{3m}{4} \rceil +1), \ \ \textrm{if} \ \  \frac{3m}{4} <k \leq m-2, \\ 2m=\nu_\ell, \ \ \textrm{if} \ \ k=m-1. \end{cases}
\end{equation}

{\bf Case 2: $a=1$ and $c\geq m$.} From Lemma \ref{i1} System \eqref{sist1} reads,
$$\begin{cases} 2bm+8c \geq 8k, \\ 2(b+1)(c+1) +2(3-b)(c-m+1) \leq 2(k+1). \end{cases}$$
Since $2(b+1)(c+1) +2(3-b)(c-m+1) \geq 2(c+1)$ and $c>k$ this case cannot occur.
\ \\

{\bf Case 3: $a=0$ and $c<m$.} From Lemma \ref{i0} \eqref{sist1} reads,
$$\begin{cases} 2bm+8c \geq 3m+8k, \\ (b+1)(c+1) + \lfloor \frac{b}{3} \rfloor (c+1) \leq 2(k+1). \end{cases}$$
\begin{itemize}
\item If $b=0$  then $c \geq \lceil k+\frac{3m}{8} \rceil$ and $c \leq 2k+1$.
Such a $c$ exists if and only if
$k \geq\lfloor \frac{3m}{8} \rfloor$.
For these cases, the minimum is obtained taking $c= \lceil k+\frac{3m}{8} \rceil$ and hence $\nu_s=(\lceil k+\frac{3m}{8} \rceil\ +1)$.
\item The case $b=1$ cannot occur. In fact we have $c \geq \lceil k+\frac{m}{8} \rceil$ and $2(c+1) \leq 2(k+1)$, a contradiction.

\item If $b=2$ then $c \geq \lceil k-\frac{m}{8} \rceil$ and $c \leq \lfloor \frac{2k-1}{3} \rfloor$.
Such a $c$ exists if and only if $k+\lceil -\frac{m}{8} \rceil \leq \lfloor \frac{2k+1}{3} \rfloor$.
This is equivalent to $k \leq 3 \lfloor \frac{m}{8} \rfloor -1$ if $2k\equiv1\pmod3$, to $k \leq 3 \lfloor \frac{m}{8} \rfloor +1$ if $2k\equiv2\pmod3$, to $k \leq 3 \lfloor \frac{m}{8} \rfloor -3$ if $2k\equiv0\pmod3$.
For these cases, the minimum is obtained taking $c=\max\{0, \lceil k-\frac{m}{8} \rceil\}$ and hence $\nu_s=3(\max\{0, \lceil k-\frac{m}{8} \rceil\} +1)$.

\item If $b=3$ then $c \geq \lceil k-\frac{3m}{8} \rceil$ and $c \leq \lfloor \frac{2k-3}{5} \rfloor$.
Such a $c$ exists if and only if $k+\lceil -\frac{3m}{8} \rceil \leq \lfloor \frac{2k-3}{5} \rfloor$.
This is equivalent to $k \leq \frac{5}{3} \lfloor \frac{3m}{8} \rfloor -\frac{5}{3}$ if $2k\equiv0\pmod5$, to $k \leq \frac{5}{3} \lfloor \frac{3m}{8} \rfloor -2$ if $2k\equiv1\pmod5$, to $k \leq \frac{5}{3} \lfloor \frac{3m}{8} \rfloor -\frac{7}{3}$ if $2k\equiv2\pmod5$, to $k \leq \frac{5}{3} \lfloor \frac{3m}{8} \rfloor -1$ if $2k\equiv3\pmod5$, to $k \leq \frac{5}{3} \lfloor \frac{3m}{8} \rfloor -\frac{4}{3}$ if $2k\equiv4\pmod5$.
In these cases, the minimum is obtained taking $c=\max\{0, \lceil k-\frac{3m}{8} \rceil\}$ and hence $\nu_s=5(\max\{0, \lceil k-\frac{3m}{8} \rceil\} +1)$.
\end{itemize}
Thus, we obtain
\begin{equation}
\label{min3}
\min\{\nu_s \mid a=0 \ \ \textrm{and} \ \ c<m\}= \begin{cases} 3, \ \ \textrm{if} \ \ k \leq \lfloor \frac{m}{8} \rfloor, \\ 5, \ \ \textrm{if} \ \ \frac{m}{8} < k \leq \lfloor \frac{3m}{8} \rfloor, \\ 5(\lceil k-\frac{3m}{8} \rceil +1), \ \textrm{if} \ \ \lceil \frac{3m}{8} \rceil \leq k \leq \lfloor \frac{5}{3} \lfloor \frac{3m}{8} \rfloor-\frac{7}{3} \rfloor, \\ (\lceil k+\frac{3m}{8} \rceil\ +1), \ \textrm{otherwise}.\end{cases}
\end{equation}

{\bf Case 4: $ a=0$ and $m \leq c < 2m$.} From Lemma \ref{i0} System \eqref{sist1} reads,
$$\begin{cases} 2bm+8c \geq 3m+8k, \\ (b+1)(c+1) + \lfloor \frac{b}{3} \rfloor (c+1) +(5-2\max\{0,b-2\})(c-m+1) \leq 2(k+1). \end{cases}$$
Since $(b+1)(c+1) + \lfloor \frac{b}{3} \rfloor (c+1) +(5-2\max\{0,b-2\})(c-m+1) \geq (b+1)(c+1)$ and $c>k$, cases $b=1,2,3$ cannot occur. Thus $b=0$ and
$$\begin{cases} 8c \geq 3m+8k, \\ (c+1) +5(c-m+1) \leq 2(k+1). \end{cases}$$
Hence $c \geq \lceil k+\frac{3m}{8}\rceil$ and $c \leq \lfloor \frac{2k+5m-4}{6} \rfloor$. Since $c \geq m$ then $k \geq \frac{m+4}{2}$. The minimum value is obtained (when it is possible) for $c = \lceil k+\frac{3m}{8}\rceil$. By direct checking the minimum $\nu_\ell$ is bigger than the one obtained in \eqref{min1}, and hence we can discard this case.
\ \\

{\bf Case 5: $a=0$ and $c\geq 2m$.} From Lemma \ref{i0} System \eqref{sist1} reads,
$$\begin{cases} 2bm+8c \geq 3m+8k, \\ (b+1)(c+1) + \lfloor \frac{b}{3} \rfloor (c+1) +(5-2\max\{0,b-2\})(c-m+1)+\max\{0,2-b\}(c-2m+1) \leq 2(k+1). \end{cases}$$
Since $(b+1)(c+1) + \lfloor \frac{b}{3} \rfloor (c+1) +(5-2\max\{0,b-2\})(c-m+1)+\max\{0,2-b\}(c-2m+1) \geq (b+1)(c+1) \geq 2m+1$ and $2(k+1) \leq 2m$ this case cannot occur.

Taking the minimum of the values in \eqref{min1} and \eqref{min3} the claim follows.
\end{proof}


Using the same arguments the following  results are obtained.

\begin{lemma} \label{dord3}
If $\rho_{\ell+1}=(1,3,k)$ for some $k<m$ then
$$d_{ORD}(C_\ell(P_\infty))=\nu_\ell.$$
\end{lemma}


\begin{lemma} \label{dord2}
If $\rho_{\ell+1}=(1,1,k)$ for some $k<m$ then
$$d_{ORD}(C_\ell(P_\infty))=\begin{cases} 4, \ \ \textrm{if} \ \ k=0, \\ 5, \ \ \textrm{if} \ \ k \leq \lfloor \frac{m}{8} \rfloor, \\ 6 \ \ \textrm{if} \ \ \frac{m}{8} < k \leq \lfloor \frac{m}{4} \rfloor, \\ 8, \ \ \textrm{if} \ \ \frac{m}{4} < k \leq \lfloor \frac{m}{2} \rfloor, \\ 8(\lceil k -\frac{m}{2} \rceil+1), \ \ \textrm{if} \ \ \lceil \frac{m}{2} \rceil \leq k\leq \lfloor \frac{3m}{4} \rfloor-2 ,  
\\ 2(\lceil \frac{m}{4} +k \rceil+1)+6(\lceil \frac{m}{4} +k \rceil-m+1), \ \ \textrm{if} \ \ \lfloor \frac{3m}{4} \rfloor-1 \leq k \leq m-2, \\ 4m, \ \ \textrm{if} \ \ k=m-1. \end{cases}$$
\end{lemma}


\begin{lemma} \label{dord4}
If $\rho_{\ell+1}=(1,2,k)$ for some $k<m$ then
$$d_{ORD}(C_\ell(P_\infty))=\begin{cases} 6, \ if \ k=0, \\ 8, \ \ \textrm{if} \ \ k \leq \lfloor \frac{m}{4} \rfloor, \\ 8(\lceil k -\frac{m}{4} \rceil+1), \ \ \textrm{if} \ \ \lceil \frac{m}{4} \rceil \leq k \leq \lfloor \frac{m}{2} \rfloor -2,
\\  2(\lceil k +\frac{m}{2} \rceil+1)+6(\lceil k +\frac{m}{2} \rceil-m+1), \ \ \textrm{if} \ \ \lfloor \frac{m}{2} \rfloor -1 \leq k \leq \lfloor \frac{3m}{4} \rfloor -2,
\\  4(\lceil k +\frac{m}{4} \rceil+1)+4(\lceil k +\frac{m}{4} \rceil-m+1), \ \ \textrm{if} \ \ \lfloor \frac{3m}{4} \rfloor -1 \leq k \leq m-2,
\\ \nu_\ell=6m, \ \ \textrm{if} \ \ k= m-1. \end{cases}$$
\end{lemma}


\subsection{Computation of $ d_{ORD}(C_{\ell}(P_\infty))$ for $ \rho_{\ell+1}=(0,j,k)$ and $\rho_{\ell} \leq 2g$}\label{Sec:d_ORD_sub3}
Using the same arguments as above we obtain the following results in the case $\rho_{\ell} \leq 2g$. 

 \begin{lemma} \label{dord5}
If $\rho_{\ell+1}=(0,0,k)$ for some $k<m$ then
$$d_{ORD}(C_\ell(P_\infty))=\begin{cases} 2, \ \ \textrm{if} \ \ k \leq \lfloor \frac{3m}{8} \rfloor, \\ 3, \ \ \textrm{if} \ \ \lceil \frac{3m}{8} \rceil \leq k \leq \lfloor \frac{m}{2} \rfloor, \\ 4, \ \ \textrm{if} \ \ \lceil \frac{m}{2} \rceil \leq k \leq \lfloor \frac{5m}{8} \rfloor, \\ 5, \ \ \textrm{if} \ \ \lceil \frac{5m}{8} \rceil \leq k \leq \lfloor \frac{3m}{4} \rfloor, \\ 6, \ \ \textrm{if} \ \ \lceil \frac{3m}{4} \rceil \leq k \leq \lfloor \frac{7m}{8} \rfloor, \\ 8, \ \ \textrm{if} \ \ \lceil \frac{7m}{8} \rceil \leq k \leq m-1. \end{cases}$$
\end{lemma}

\begin{lemma} \label{dord6}
If $\rho_{\ell+1}=(0,1,k)$ for some $k<m$ then
$$d_{ORD}(C_\ell(P_\infty))=\begin{cases} 2, \ \ \textrm{if} \ \ k \leq \lfloor \frac{m}{8} \rfloor, \\ 3, \ \ \textrm{if} \ \ \lceil \frac{m}{8} \rceil \leq k \leq \lfloor \frac{m}{4} \rfloor, \\ 4, \ \ \textrm{if} \ \ \lceil \frac{m}{4} \rceil \leq k \leq \lfloor \frac{3m}{8} \rfloor, \\ 5, \ \ \textrm{if} \ \ \lceil \frac{3m}{8} \rceil \leq k \leq \lfloor \frac{m}{2} \rfloor, \\ 6, \ \ \textrm{if} \ \ \lceil \frac{m}{2} \rceil \leq k \leq \lfloor \frac{5m}{8} \rfloor, \\ 8(\max\{0,\lceil k-\frac{7m}{8} \rceil\}+1), \ \ \textrm{if} \ \ \lceil \frac{5m}{8} \rceil \leq k \leq m-1. \end{cases}$$
\end{lemma}

\begin{lemma}
If $\rho_{\ell+1}=(0,3,k)$ for some $k<m$ then
$$d_{ORD}(C_\ell(P_\infty))=\begin{cases} 6, \ \ \textrm{if} \ \ k \leq \lfloor \frac{m}{8} \rfloor, \\ 8(\max\{0,\lceil \frac{3m}{8} \rceil\}+1), \ \ \textrm{if} \ \ \lceil \frac{m}{8} \rceil \leq k \leq m-2, \\ \nu_\ell=5(k+1), \ \ \textrm{if} \ \ k=m-1. \end{cases}$$
\end{lemma}

\begin{lemma} \label{dord7}
If $\rho_{\ell+1}=(0,2,k)$ for some $k<m$ then
$$d_{ORD}(C_\ell(P_\infty))=\begin{cases} 4, \ \ \textrm{if} \ \ k \leq \lfloor \frac{m}{8} \rfloor, \\ 5, \ \ \textrm{if} \ \ \lceil \frac{m}{8} \rceil \leq k \leq \lfloor \frac{m}{4} \rfloor, \\ 6, \ \ \textrm{if} \ \ \lceil \frac{m}{4} \rceil \leq k \leq \lfloor \frac{3m}{8} \rfloor, \\ 8(\max\{0, \lceil k-\frac{5m}{8}\rceil \}+1), \ \ \textrm{if} \ \ \lceil \frac{3m}{8} \rceil \leq k \leq \lfloor \frac{7m}{8} \rfloor-2, \\ 2(\lceil k+\frac{m}{8} \rceil+1), \ \ \textrm{if} \ \ \lfloor \frac{7m}{8} \rfloor -1 \leq k \leq m-3, \\ 3(k+1)=\nu_\ell, \ \ \textrm{if} \ \ k  \in \{m-2, m-1\}. \end{cases}$$
\end{lemma}

\begin{lemma} \label{dord8}
If $\rho_{\ell+1}=(0,0,k)$ for some $m \leq k <2m$ then
$$d_{ORD}(C_\ell(P_\infty))=\begin{cases} 8(\lceil k-\frac{9m}{8}\rceil+1), \ \ \textrm{if} \ \ m \leq k < \lfloor\frac{11m}{8} -1 \rfloor, \\
2(\lceil k-\frac{3m}{8} \rceil+1)+\max\{0, 6(\lceil k-\frac{3m}{8}\rceil -m+1)\}, \ \ \textrm{if} \ \ \lfloor\frac{11m}{8} -1 \rfloor \leq k <2m. \end{cases}$$
\end{lemma}

\subsection{Computation of $d_{ORD}(C_{\ell}(P_\infty))$ for $\rho_{\ell}>2g$} \label{Sec:d_ORD_sub4}

\begin{proposition}\label{rimarco}
The Weierstrass semigroup $H(P_\infty)=\langle q^3, m q, q^n+1 \rangle$ is telescopic.
\end{proposition}

\begin{proof}
Let $a_1=q^3$, $a_2=mq$, $a_3=q^n+1$, $d_0=0$, $d_1=q^3$, $d_2=\gcd(q^3,mq)=q$, $d_3=\gcd(q^3,mq,q^n+1)=1$. Then $a_i/d_i\in\langle a_1/d_{i-1},\ldots,a_{i-1}/d_{i-1} \rangle$ for $i=2,3$; that is, $H(P_\infty)$ is telescopic.
\end{proof}

Proposition \ref{rimarco} implies that $H(P_\infty)$ is symmetric, from \cite[Lemma 6.5]{KP}.
This also follows from the fact that the divisor $(2g-2)P_\infty$ is canonical; see \cite[Lemma 3.8]{GOS} and \cite[Remark 4.4]{KP}.

In the following, Proposition \ref{campillo} is used to reduce the direct computation of $d_{ORD}(C_\ell(P_\infty))$ with $\rho_{\ell}>2g$, only to those cases for which $\rho_{\ell+1}\ne2g-1+e$ for any $e \in H(P_\infty) \setminus \{0\}$. Since the cases in which $\rho_{\ell+1}=(0,0,k)$ for $k<2m$ or $\rho_{\ell+1}=(i,j,k)$ for $k<m$ have been already studied, they can be excluded.

\begin{proposition} \label{farran} Let $\rho_{\ell} \in H(P_\infty)$ with $\rho_{\ell}>2g$ and $\rho_{\ell+1}=(i,j,k)$ and $k \geq m$. If $\rho_{\ell+1} \ne (0,0,k)$ for any $k \in [m,2m)$, then $\rho_{\ell+1}-2g+1 \not\in H(P_\infty)$ if and only if $\rho_{\ell+1}=(0,1,k)$ for some $k \in [m,2m)$.\end{proposition}

\begin{proof}
Write $k=m+s$ for some $s \geq 0$. We prove the claim using a case-by-case analysis with respect to the values of $i$ and $j$. We recall that $2g-1=(2^{n+1}+2^{n}-4)-1=9m-8$.
\ \\ \\
{\bf Case 1}: $i=1$. Clearly, $\rho_{\ell+1}=3m+2jb+8m+8s$.
\begin{itemize}
\item If $j=0$, then $\rho_{\ell+1}=3m+8m+8s=(9m-8)+(2m+8(s+1))=2g-1+e$. Writing $e=(0,1,s+1)$ we have that $e \in H(P_\infty)$, so this case cannot occur.
\item If $j=1$, then $\rho_{\ell+1}=3m+2m+8m+8s=(9m-8)+(4m+8(s+1))=2g-1+e$. Writing $e=(0,2,s+1)$ we have that $e \in H(P_\infty)$, so this case cannot occur.
\item If $j=2$, then $\rho_{\ell+1}=3m+4m+8m+8s=(9m-8)+(6m+8(s+1))=2g-1+e$. Writing $e=(0,3,s+1)$ we have that $e \in H(P_\infty)$, so this case cannot occur.
\item If $j=3$, then $\rho_{\ell+1}=3m+6m+8k=(9m-8)+(8(k+1))=2g-1+e$. Writing $e=(0,0,k+1)$ we have that $e \in H(P_\infty)$, so this case cannot occur.
\end{itemize}
\ \\ 
{\bf Case 2}: $i=0$. Clearly, $\rho_{\ell+1}=2jb+8k=2jb+8m+8s$.
\begin{itemize}
\item If $j=0$, then in particular we can write $k=2m+t$ for $t \geq 0$, since $k \geq 2m$. Thus,  $\rho_{\ell+1}=16m+8t=(9m-8)+(7m+8(t+1))=2g-1+e$. Writing $e=(1,2,t+1)$ we have that $e \in H(P_\infty)$, so this case cannot occur.

\item If $j=1$, then $\rho_{\ell+1}=2m+8k$. We first assume that $k \geq 2m$ and so that $k=2m+t$ for some $t \geq 0$. In this case $\rho_{\ell+1}=2m+16m+8t=(9m-8)+(9m+8(t+1))=2g-1+e$. Writing $e=(1,3,t+1)$ we have that $e \in H(P_\infty)$, so this case cannot occur. Thus, $k \in [m,2m)$. In this case, $\rho_{\ell+1}=2m+8m+8s=(9m-8)+(m+8(s+1))=2g-1+e$. By direct computation $e \not\in H(P_\infty)$ and the claim follows.

\item If $j=2$, then $\rho_{\ell+1}=4m+8m+8s=(9m-8)+(3m+8(s+1))=2g-1+e$. Writing $e=(1,0,s+1)$ we have that $e \in H(P_\infty)$, so this case cannot occur.

\item If $j=3$, then $\rho_{\ell+1}=6m+8m+8s=(9m-8)+(5m+8(s+1))=2g-1+e$. Writing $e=(1,1,s+1)$ we have that $e \in H(P_\infty)$, so this case cannot occur.
\end{itemize}
\end{proof}

Since from Proposition \ref{rimarco} the Weierstrass semigroup $H(P_\infty)$ is symmetric, its conductor is $c=2g$; equivalently, its largest gap is $2g-1$. The following theorem shows that the exact value of $d_{ORD}(C_{\ell}(P_\infty))$ is known for $\rho_{\ell+1} \geq 4g$; see \cite[Proposition 4.2 (iii)]{CF}.

\begin{theorem} \label{fengrao}
Let $H(P)$ be a Weierstrass semigroup. Then $d_{ORD}(C_{\ell}(P)) \geq \ell+1-g$ and equality holds if $\rho_{\ell+1} \geq 4g$.
\end{theorem}

According to the results obtained in the previous sections,  Remark \ref{farran}, and Theorem \ref{fengrao}, to complete the computation of $d_{ORD}(C_\ell(P_\infty))$ for every $\rho_{\ell} \in H(P_\infty)$, only the case $\rho_{\ell} \in [2g,4g-1)$ with $\rho_{\ell+1}=(0,1,k)$ and $k \in [m,2m)$ has to be considered. 


\begin{proposition} \label{perquant} Let $\rho_\ell \in H(P_\infty)$ be such that $\rho_{\ell}>2g$ and $\rho_{\ell+1}=(0,1,k)<4g$ for $k \in [m,2m)$. Then $$d_{ORD}(C_{\ell}(P_\infty))=\begin{cases} \nu_{\ell+5}=8k-7m+13, \ \ \textrm{if} \ \ k < \frac{9m-11}{8}, \\ \nu_{\ell+3}=8k-7m+11, \ \ \textrm{if} \ \ \frac{9m-11}{8} \leq k < \frac{11m-9}{8}, \\ \nu_{\ell+1}=8k-7m+9, \ \ \textrm{if} \ \ k \geq \frac{11m-9}{8}.\end{cases} $$ \end{proposition}

\begin{proof} Arguing as in the previous propositions one can prove that the value of $d_{ORD}(C_{\ell}(P_\infty))$ is obtained by $\nu_{\ell+5}$, $\nu_{\ell+3}$, and $\nu_{\ell+1}$, if $k < \frac{9m-11}{8}$, $\frac{9m-11}{8} \leq k < \frac{11m-9}{8}$, and $k \leq \frac{11m-9}{8}$ respectively. Since $\rho_{\ell+1} \geq 2g$, we have that $\rho_{\ell+t}=\rho_{\ell+1}+(t-1)$ for every $t \geq 1$. 

Assume that $k<\frac{9m-11}{8}$. By direct checking $\rho_{\ell+6}=\rho_{\ell+1}+5=(1,3,\tilde k)$, where $\tilde k =k-\frac{7m-5}{8}$. Hence from Lemma \ref{i1},  $d_{ORD}(C_{\ell}(P_\infty))=8(k-\frac{7m-5}{8}+1)=8k-7m+13$, as $\tilde k < \frac{9m-11}{8}- \frac{7m-5}{8}<m$. 

Assume that $\frac{9m-11}{8} \leq k < \frac{11m-9}{8}$. By direct checking $\rho_{\ell+4}=\rho_{\ell+1}+3=(1,0,\tilde k)$ where $\tilde k =k-\frac{m-3}{8}$. Hence $\tilde k \geq m-1$ and from Lemma \ref{i1}, $d_{ORD}(C_{\ell}(P_\infty))=2m=8k-7m+11$ if $\tilde k =m-1$, while $d_{ORD}(C_{\ell}(P_\infty))=2(\tilde k+1)+6(\tilde k -m+1)=8k-7m+11$ if $\tilde k \geq m$. 

Assume that $k \geq \frac{11m-9}{8}$. By direct checking $\rho_{\ell+2}=\rho_{\ell+1}+1=(1,1,\tilde k)$ where $\tilde k =k-\frac{3m-1}{8}$. Hence $\tilde k \geq m-1$ and from Lemma \ref{i1}, $d_{ORD}(C_{\ell}(P_\infty))=4m=8k-7m+9$ if $\tilde k =m-1$, while $d_{ORD}(C_{\ell}(P_\infty))=4(\tilde k+1)+4(\tilde k -m+1)=8k-7m+9$ if $\tilde k \geq m$.\end{proof}

For $q\ne2$, we cannot determine $d_{ORD}(C_{\ell}(P_\infty)$ for all $\ell$. Yet, this is possible for certain $\ell$, as shown in the following propositions.

\begin{proposition}
If $\rho_{\ell+1}\leq (q-1)(q^n+1)$,
then
$$ d_{ORD}(C_{\ell}(P_\infty)) = j+1, $$
where $j\leq q-1$ satisfies $(j-1)(q^n+1) < \rho_{\ell+1} \leq j(q^n+1)$.
\end{proposition}
\begin{proof}
Since $H(P_\infty)$ is telescopic from Proposition \ref{rimarco}, we can apply \cite[Theorem 6.11]{KP}. The claim then follows because $q^n+1=\max\{\frac{q^3}{1},\frac{mq}{1},\frac{q^n+1}{1}\}$.
\end{proof}

\begin{proposition}
If $\frac{3}{2}(q-1)(q^{n+1}+\frac{1}{3}q^n-q^2-\frac{2}{3})-2<\ell\leq \frac{3}{2}(q-1)(q^{n+1}+q^n-q^2)-2$,
then
$$ d_{ORD}(C_{\ell}(P_\infty)) = \min\{\rho_t\mid \rho_t\geq \ell+1-g\}. $$
\end{proposition}

\begin{proof}
This is the claim of \cite[Theorem 6.10]{KP}.
\end{proof}

\subsection{Application for $q=2$: families of AG codes with relative Singleton defect going to zero}\label{Sec:Application1}

In this section, we assume that $q=2$ and provide two families of codes of type $C_{\ell}(P_\infty)$ in the cases $\rho_\ell=9m$ and $\rho_\ell=9m+8$, with relative Singleton defect going to zero as $n$ goes to infinity.
We denote by $\delta$ and $\Delta$ the Singleton defect and the relative Singleton defect of $C_\ell(P_\infty)$, respectively.
\



\begin{lemma}
Fix $n\geq5$ odd. Then $9m-1, \ 9m, \ 9m+1 \in H(P_\infty)$.
\end{lemma}

\begin{proof}
A direct computation shows that $9m-1=(0,3,\frac{2^n}{8})$, $9m=(1,3,0)$, and $9m+1=(0,2,\frac{5\cdot2^{n-3}+1 }{3})$, thus the claim follows.
\end{proof}

We now assume that $\rho_\ell=9m$. Since $\rho_{\ell+1}=9m+1=(0,2,\frac{5\cdot2^{n-3}+1 }{3})$ the following result follows from Lemma 1.3.
\begin{corollary} Assume that $\rho_\ell=(1,3,0)$. Then $\nu_\ell=3 \cdot \big(\frac{5 \cdot 2^{n-3}+1}{3}+1 \big) \geq 24$.
\end{corollary}

\begin{proposition} \label{9m} The code $C_\ell(P_\infty)$ is an $[N,k,d]_{2^{2n}}$-linear code with
\begin{itemize}
\item $N=(3m-1)^2+(3m-1)(9m-7)$,
\item $k=N-\frac{9m+9}{2}$,
\item $d \geq d_{ORD}(C_\ell(P_\infty)) = 16$,
\item $\delta \leq N-k+1- d_{ORD}(C_\ell(P_\infty))=\frac{9m-21}{2}$,
\item $\Delta =\frac{\delta}{N} \leq \frac{9m-21}{2(3m-1)(12m-8)}$; hence, $\Delta$ goes to zero as $n$ goes to infinity.
\end{itemize}
\end{proposition}

\begin{proof}
Since $GGS(q,n)$ is $\mathbb{F}_{2^{2n}}$-maximal, we have
$$N=(2^{2n}+1+2g(GGS(q,n))2^n)-1=2^{2n}+2^n(9m-7)=(3m-1)^2+(3m-1)(9m-7);$$
the last equality follows from $m=(2^n+1)/3$.
Since $C_\ell(P_\infty)=C^{\perp}(\overline D,\rho_\ell P_\infty)$, $k=N-\tilde{k}$ where $\tilde{k}$ is the dimension of $C(\overline D,\rho_\ell P_\infty)$. As $\deg(\rho_\ell P_\infty)>2g(GGS(q,n))-2$, from the Riemann-Roch Theorem follows
$$k=N-\tilde{k}=N-(\deg(\rho_\ell P_\infty)+1-g(GGS(q,n)))=N-\left(9m+1-\frac{9m-7}{2}\right)=r-\frac{9m+9}{2}.$$

By Lemma \ref{dord7}, $d_{ORD}(C_\ell(P_\infty))\geq 16$. To prove the claim is sufficient to show that there exists $\rho_s \geq \rho_\ell$ such that $\nu_s=16$. To this end we take $\rho_{s+1}=(1,3,1)=9m+8>9m+1$. From Lemma 1.1, $\nu_s=2(b+1)(c+1)=2(3+1)(1+1)=16$ and the claim follows. Now the claim on $\delta$ and $\Delta$ follows by direct computation.
\end{proof}

We now assume that $\rho_\ell=9m+8$, so that $\rho_{\ell+1}=9m+9=(0,2,\frac{5\cdot2^{n-3}+1}{3}+1)=(0,2,\frac{5m+9}{8})$. 
Arguing as in the proof of Proposition \ref{9m} and using Lemma \ref{dord7}, the following result is obtained.

\begin{proposition} The code $C_\ell(P_\infty)$ is an $[r,k,d]_{2^{2n}}$-linear code with
\begin{itemize}
\item $r=(3m-1)^2+(3m-1)(9m-7)$,
\item $k=r-\frac{9m+25}{2}$,
\item $d \geq d_{ORD}(C_\ell(P_\infty)) = 2 \big(\lceil \frac{6m+9}{8} \rceil +1 \big)$,
\item $\delta \leq r-k+1- d_{ORD}(C_\ell(P_\infty))=\frac{9m+25}{2} -2 \big(\lceil \frac{6m+9}{8} \rceil +1 \big) <\frac{6m+21}{2}$,
\item $\Delta =\frac{\delta}{r}
< \frac{6m+21}{2(3m-1)(12m-8)}$; hence, $\Delta$ goes to zero as $n$ goes to infinity.
\end{itemize}
\end{proposition}

\section{Weierstrass semigroup at $P_0$}\label{Sec:P_0}

In this section we describe the Weierstrass semigroup at $P_0$, and hence at any $\mathbb F_{q^2}$-rational affine point by Lemma \ref{Orbite}.
Consider the functions 
\begin{equation}\label{functions}
\frac{y^r z^t}{x^s}, \qquad s\in [0,q^2-1], r\in [0,s], t\in \left[0,\left\lfloor\frac{sm(q+1)-rqm}{q^3} \right \rfloor \right].
\end{equation}
All these functions belong to $H(P_0)$. In fact,
$$\left(\frac{y^r z^t}{x^s}\right)=(mr+t-m(q+1)s)P_0+(m(q+1)s-mqr-tq^3)P_{\infty}$$
and by assumption 
$$m(q+1)s-mqr-tq^3\geq 0.$$

\begin{proposition}
Let $t \in  \left[0,\min\left(\left\lfloor\frac{sm(q+1)-rqm}{q^3} \right \rfloor ,m-1\right)\right]$ and 
$$s\in [0,q], r\in [0,s]$$
or 
$$s\in [q+1,q^2-1], r\in [0,q].$$
Then all the integers $mr+t-m(q+1)s$ are distinct.
\end{proposition}
\proof
Suppose $mr+t-m(q+1)s=m\overline r+\overline t-m(q+1)\overline s$. Then $t\equiv \overline t \pmod{m}$, which implies $t=\overline t$. Now, from $mr-m(q+1)s=m\overline r-m(q+1)\overline s$, $r\equiv \overline r \pmod{q+1}$, which yields $r=\overline r$ and $s=\overline s$. 
\endproof

\begin{proposition}
Consider the following sets 
$$\begin{array}{l}
\mathcal{L}_1 := \Big\{-t-rm+m(q+1)s \mid s\in [0,q], r\in [0,s], \\
 \hspace{8 cm} t\in \left[0,((s-r)q+s)\frac{m-q^2+q-1}{q^3}+s-r \right]\Big\};\\
\\
\mathcal{L}_2 := \Big\{-t-rm+m(q+1)s \mid s\in [q+1,q^2-q], r\in [0,q],\\
 \hspace{8 cm} t\in \left[0,((s-r)q+s)\frac{m-q^2+q-1}{q^3}+s-r \right]\Big\};\\
 \\
\mathcal{L}_3 := \Big\{-t-rm+m(q+1)s \mid s\in [q^2-q+1,q^2-2],\\
 \hspace{8.5 cm} r\in [0,q+s-q^2-1], t\in \left[0,m-1 \right]\Big\};\\
\\
\mathcal{L}_4 := \Big\{-t-rm+m(q+1)s \mid s\in [q^2-q+1,q^2-2], r\in [q+s-q^2,q],\\
 \hspace{8 cm} t\in \left[0,((s-r)q+s)\frac{m-q^2+q-1}{q^3}+s-r \right]\Big\};\\
 \\
\mathcal{L}_5 := \left\{-t+m(q+1)(q^2-1) \mid  t\in \left[q^3,m-1 \right]\right\};\\
\\
\mathcal{L}_6 := \left\{-t-rm+m(q+1)(q^2-1) \mid r\in [1,q-2], t\in \left[0,m-1 \right]\right\};\\
\\
\mathcal{L}_7 := \Big\{-t-rm+m(q+1)(q^2-1) \mid r\in [q-1,q], \\
 \hspace{6.5 cm} t\in \left[0,((q^2-1-r)q+q^2-1)\frac{m-q^2+q-1}{q^3}+q^2-1-r \right]\Big\}.\\
\end{array}
$$
Then each $\mathcal{L}_i$ is contained in $\mathcal{L}((2g-1)P_0)$.
\end{proposition}
\proof
By direct computations.
\endproof

Finally, we can give the description of the Weierstrass semigroup $H(P_0)$.

\begin{proposition}\label{SemigruppoP0}
$$\bigcup _{i=1}^7\mathcal{L}_i =H(P_0) \cap \{0,\ldots, 2g-1\}.$$
\end{proposition}
\proof
By direct computations, since 
$$|\mathcal{L}_1|=\left(\frac{q^4 + 5q^3 + 8q^2 + 4q}{6}\right)\left(\frac{m-q^2+q-1}{q^3}\right)+\frac{(q+1)(q+2)(q+3)}{6},$$
$$|\mathcal{L}_2|=\left(\frac{q^6 - q^5 - q^4 - 3q^2 - 2q}{2}\right)\left(\frac{m-q^2+q-1}{q^3}\right)+\frac{q^5 - 2q^4 + 2q^3 - q^2 - 6q}{2},$$
$$|\mathcal{L}_4|=\left(\frac{3q^5 + 2q^4 - 20q^3 + q^2 + 8q + 12}{6}\right)\left(\frac{m-q^2+q-1}{q^3}\right)+\frac{3q^4 - q^3 - 18q^2 + 22q - 12}{6},$$
$$|\mathcal{L}_3|=\frac{m(q-2)(q-1)}{2}, \qquad |\mathcal{L}_5|=m-q^3, \qquad |\mathcal{L}_6|=\sum_{r=1}^{q-2} m=(q-2)m,$$
$$|\mathcal{L}_7|=\frac{m-q^2+q-1}{q^3}(2q^3-q-2)+2q^2-2q+1$$
and $\mathcal{L}_i \cap \mathcal{L}_j=\emptyset $ if $i \neq j$.
\endproof

Let $C_\ell(P_0)=C^{\perp}(\tilde D,\tilde\rho_\ell P_0)$, where $\tilde D$ is as in \eqref{Dbarra} and $\tilde\rho_\ell$ is the $\ell$-th positive non-gap at $P_0$.
In this case it has not been possible to determine $d_{ORD}(C_\ell(P_0))$ for any $n$ since we do not have a basis of the Weierstrass semigroup at $P_0$.
Nevertheless, Tables \ref{tabella1}, \ref{TabMeglio}, and \ref{tabquantum1} give evidence that for some specific values of $\ell$ the AG codes and AG quantum codes from $C_\ell(P_0)$ are better than $C_\ell(P_\infty)$, since the designed relative Singleton defect of $C_\ell(P_0)$ is smaller than the one of $C_\ell(P_\infty)$.

\section{AG codes on the GGS curve for $q=2$ and $n=5$}\label{Sec:tabelle}

In this section a more detailed description of the results obtained in the previous sections is given for the particular case $q=2$, $n=5$. Recall that in this case
$$ H(P_\infty)=\{0,8,16,22,24,30,32,33,38,40,41,44,46,48,49,52\}\cup \{54,\ldots,57\} $$
$$ \cup\,\{60\}\cup\{62,\ldots, 66\} \cup\{68\}\cup\{70,\ldots, 74\}\cup\{76,\ldots,82\}\cup\{84,\ldots,90\}\cup\{92,\ldots\}. $$

For the point $P_0$ (and hence for any $\mathbb F_{q^2}$-rational point), we have from Proposition \ref{SemigruppoP0}
$$ H(P_0)=\{0,21,22\}\cup\{29,\ldots,33\}\cup\{42,43,44\}\cup\{50,\ldots,55\}\cup\{58,\ldots,66\}\cup\{71,\ldots,77\}\cup\{79,\ldots\}. $$

Table \ref{tabella1} contains the parameters of the codes $C_{\ell_\infty}(P_\infty)$ and $C_{\ell_0}(P_0)$; in particular, their common length $N=3968$ and dimension $k$, their Feng-Rao designed minimum distance $d_{ORD}^{\infty}$ and $d_{ORD}^{0}$, their designed Singleton defects $\delta_\infty=N+1-k-d_{ORD}^{\infty}$ and $\delta_0=N+1-k-d_{ORD}^{0}$, and their designed relative Singleton defects $\Delta_{\infty}=\frac{\delta_{\infty}}{N}$ and $\Delta_{0}=\frac{\delta_{0}}{N}$.

\begin{center}
\tabcolsep = 0.5 mm
{\scriptsize
\begin{longtable}{|c|c||c|c|c|c||c|c|c|c|}
\caption[Codes $C_{\ell_\infty}(P_\infty)$ and $C_{\ell_0}(P_0)$, $q^n=2^5$]{Codes $C_{\ell_\infty}(P_\infty)$ and $C_{\ell_0}(P_0)$, $q^n=2^5$} \label{tabella1} \\
\hline
  &   &   &    &   &   &   &   &   &  \\[-2 mm]
   $N$ & $k$ & $\rho_{\ell_\infty}$ & $d_{ORD}^{\infty}$  & $\delta_\infty \leq$ & $\Delta_\infty \leq$ & $\rho_{\ell_0}$ & $d_{ORD}^{0}$ & $\delta_0 \leq$ & $\Delta_0 \leq$\\
  &   &   &    &   &   &   &   &   &  \\[-2 mm]

\endfirsthead

\multicolumn{10}{c}%
{{ \tablename\ \thetable{} : continued from previous page}} \\
\hline
  &   &   &    &   &   &   &   &   &  \\[-2 mm]
$n$ & $k$ & $\rho_{\ell_\infty}$ & $d_{ORD}^{\infty}$  & $\delta_\infty \leq$ & $\Delta_\infty \leq$ & $\rho_{\ell_0}$ & $d_{ORD}^{0}$ & $\delta_o \leq$ & $\Delta_0 \leq$\\
  &   &   &    &   &   &   &   &   &  \\[-2 mm]
\endhead

\hline
\endfoot

\hline
\endlastfoot

\hline 3968 & 3966 & 8 & 2 & 1 & 0,0003 & 21 & 2 & 1 & 0,0003\\
\hline 3968 & 3965 & 16 & 2 & 2 & 0,0006 & 22 & 2 & 2 & 0,0006\\
\hline 3968 & 3964 & 22 & 2 & 3 & 0,0008 & 29 & 2 & 3 & 0,0008\\
\hline 3968 & 3963 & 24 & 2 & 4 & 0,0011 & 30 & 2 & 4 & 0,0011\\
\hline 3968 & 3962 & 30 & 2 & 5 & 0,0013 & 31 & 2 & 5 & 0,0013\\
\hline 3968 & 3961 & 32 & 2 & 6 & 0,0016 & 32 & 2 & 6 & 0,016\\
\hline 3968 & 3960 & 33 & 3 & 6 & 0,0016 & 33 & 3 & 6 & 0,016\\
\hline 3968 & 3959 & 38 & 3 & 7 & 0,0018 & 42 & 3 & 7 & 0,018\\
\hline 3968 & 3958 & 40 & 3 & 8 & 0,0021 & 43 & 3 & 8 & 0,021\\
\hline 3968 & 3957 & 41 & 3 & 9 & 0,0023 & 44 & 3 & 9 & 0,023\\
\hline 3968 & 3956 & 44 & 4 & 9 & 0,0023 & 50 & 3 & 10 & 0,026\\
\hline 3968 & 3955 & 46 & 4 & 10 & 0,0026 & 51 & 3 & 11 & 0,028\\

\hline 3968 & 3954 & 48 & 4 & 11 & 0,0028 & 52 & 3 & 12 & 0,031\\
\hline 3968 & 3953 & 49 & 4 & 12 & 0,0031 & 53 & 3 & 13 & 0,033\\
\hline 3968 & 3952 & 52 & 4 & 13 & 0,0033 & 54 & 3 & 14 & 0,036\\
\hline 3968 & 3951 & 54 & 4 & 14 & 0,0036 & 55 & 3 & 15 & 0,038\\
\hline 3968 & 3950 & 55 & 5 & 14 & 0,0036 & 58 & 4 & 15 & 0,038\\
\hline 3968 & 3949 & 56 & 5 & 15 & 0,0038 & 59 & 5 & 15 & 0,038\\
\hline 3968 & 3948 & 57 & 5 & 16 & 0,0041 & 60 & 5 & 16 & 0,041\\

\hline 3968 & 3947 & 60 & 5 &  17 & 0,0043 & 61 & 5 & 17 & 0,043\\
\hline 3968 & 3946 & 62 & 5 & 18 & 0,0046 & 62 & 5 & 18 & 0,046\\
\hline 3968 & 3945 & 63 & 5 & 19 & 0,0048 & 63 & 5 & 19 & 0,048\\
\hline 3968 & 3944 & 64 & 5 & 20 & 0,0051 & 64 & 5 & 20 & 0,051\\
\hline 3968 & 3943 & 65 & 5 & 21 & 0,0053 & 65 & 5 & 21 & 0,053\\
\hline 3968 & 3942 & 66 & 6 & 21 & 0,0053 & 66 & 6 & 21 & 0,053\\
\hline 3968 & 3941 & 68 & 6 & 22 & 0,0056 & 71 & 6 & 22 & 0,056\\
\hline 3968 & 3940 & 70 & 6 & 23 & 0,0058 & 72 & 6 & 23 & 0,058\\
\hline 3968 & 3939 & 71 & 6 & 24 & 0,0061 & 73 & 6 & 24 & 0,061\\
\hline 3968 & 3938 & 72 & 6 & 25 & 0,0064 & 74 & 6 & 25 & 0,064\\
\hline 3968 & 3937 & 73 & 6 & 26 & 0,0066 & 75 & 6 & 26 & 0,066\\
\hline 3968 & 3936 & 74 & 6 & 27 & 0,0069 & 76 & 6 & 27 & 0,069\\

\hline 3968 & 3935 & 76 & 6 & 28 & 0,0071 & 77 & 6 & 28 & 0,0071\\
\hline 3968 & 3934 & 77 & 8 & 27 & 0,0069 & 79 & 8 & 27 & 0,069\\
\hline 3968 & 3933 & 78 & 8 & 28 & 0,0071 & 80 & 8 & 28 & 0,0071\\
\hline 3968 & 3932 & 79 & 8 & 29 & 0,0074 & 81 & 8 & 29 & 0,0074\\
\hline 3968 & 3931 & 80 & 8 & 30 & 0,0076 & 82 & 8 & 30 & 0,0076\\
\hline 3968 & 3930 & 81 & 8 & 31 & 0,0079 & 83 & 8 & 31 & 0,0079\\
\hline 3968 & 3929 & 82 & 8 & 32 & 0,0081 & 84 & 8 & 32 & 0,0081\\
\hline 3968 & 3928 & 84 & 8 & 33 & 0,0084 & 85 & 8 & 33 & 0,0084\\
\hline 3968 & 3927 & 85 & 8 & 34 & 0,0086 & 86 & 8 & 34 & 0,0086\\
\hline 3968 & 3926 & 86 & 8 & 35 & 0,0089 & 87 & 8 & 35 & 0,0089\\
\hline 3968 & 3925 & 87 & 8 & 36 & 0,0091 & 88 & 8 & 36 & 0,0091\\
\hline 3968 & 3924 & 88 & 8 & 37 & 0,0094 & 89 & 8 & 37 & 0,0094\\
\hline 3968 & 3923 & 89 & 8 & 38 & 0,0096 & 90 & 8 & 38 & 0,0096\\
\hline 3968 & 3922 & 90 & 8 & 39 & 0,0099 & 91 & 8 & 39 & 0,0099\\
\hline 3968 & 3921 & 92 & 8 & 40 & 0,0101 & 92 & 8 & 40 & 0,0101\\
\hline 3968 & 3920 & 93 & 8 & 41 & 0,0104 & 93 & 8 & 41 & 0,0104\\
\hline 3968 & 3919 & 94 & 8 & 42 & 0,0106 & 94 & 8 & 42 & 0,0106\\
\hline 3968 & 3918 & 95 & 8 & 43 & 0,0109 & 95 & 8 & 43 & 0,0109\\
\hline 3968 & 3917 & 96 & 8 & 44 & 0,0111 & 96 & 8 & 44 & 0,0111\\
\hline 3968 & 3916 & 97 & 8 & 45 & 0,0114 & 97 & 8 & 45 & 0,0114\\
\hline 3968 & 3915 & 98 & 8 & 46 & 0,0116 & 98 & 8 & 46 & 0,0116\\
\hline 3968 & 3914 & 99 & 16 & 39 & 0,0099 & 99 & 9 & 46 & 0,0116\\
\hline 3968 & 3913 & 100 & 16 & 40 & 0,0101 & 100 & 16 & 40 & 0,0101\\
\hline 3968 & 3912 & 101 & 16 & 41 & 0,0104 & 101 & 21 & 36 & 0,0091\\
\hline 3968 & 3911 & 102 & 16 & 42 & 0,0106 & 102 & 22 & 36 & 0,0091\\
\hline 3968 & 3910 & 103 & 16 & 43 & 0,0109 & 103 & 22 & 37 & 0,0094\\
\hline 3968 & 3909 & 104 & 16 & 44 & 0,0111 & 104 & 22 & 38 & 0,0096\\
\hline 3968 & 3908 & 105 & 16 & 45 & 0,0114 & 105 & 22 & 39 & 0,0099\\
\hline 3968 & 3907 & 106 & 16 & 46 & 0,0116 & 106 & 22 & 40 & 0,0101\\
\hline 3968 & 3906 & 107 & 22 & 41 & 0,0104 & 107 & 22 & 41 & 0,0104\\
\hline 3968 & 3905 & 108 & 22 & 42 & 0,0106 & 108 & 22 & 42 & 0,0106\\
\hline 3968 & 3904 & 109 & 22 & 43 & 0,0109 & 109 & 22 & 43 & 0,0109\\
\hline 3968 & 3903 & 110 & 22 & 44 & 0,0111 & 110 & 22 & 44 & 0,0111\\
\hline 3968 & 3902 & 111 & 22 & 45 & 0,0114 & 111 & 26 &41 & 0,0104\\
\hline 3968 & 3901 & 112 & 22 & 46 & 0,0116 & 112 & 29 & 39 & 0,0099\\
\hline 3968 & 3900 & 113 & 24 & 45 & 0,0114 & 113 & 29 & 40 & 0,0101\\
\hline 3968 & 3899 & 114 & 24 & 46 & 0,0116 & 114 & 29 & 41 & 0,0104\\
\hline 3968 & 3898 & 115 & 30 & 41 & 0,0104 & 115 & 29 & 42 & 0,0106\\
\hline 3968 & 3897 & 116 & 30 & 42 & 0,0106 & 116 & 29 & 43 & 0,0109\\
\hline 3968 & 3896 & 117 & 30 & 43 & 0,0109 & 117 & 29 & 44 & 0,0111\\
\hline 3968 & 3895 & 118 & 30 & 44 & 0,0111 & 118 & 29 & 45 & 0,0114\\
\hline 3968 & 3894 & 119 & 30 & 45 & 0,0114 & 119 & 29 & 46 & 0,0116\\
\hline 3968 & 3893 & 120 & 30 & 46 & 0,0116 & 120 & 30 & 46 & 0,0116\\
\hline 3968 & 3892 & 121 & 32 & 45 & 0,0114 & 121 & 31 & 46 & 0,0116\\
\hline 3968 & 3891 & 122 & 32 & 46 & 0,0116 & 122 & 36 & 42 & 0,0106\\
\hline 3968 & 3890 & 123 & 33 &  46 & 0,0116 & 123 & 37 & 42 & 0,0106\\

\hline 3968 & 3889 & 124 & 38 & 42 & 0,0106 & 124 & 37 & 43 & 0,0109\\
\hline 3968 & 3888 & 125 & 38 & 43 & 0,0109 & 125 & 37 & 44 & 0,0111\\
\hline 3968 & 3887 & 126 & 38 & 44 & 0,0111 & 126 & 37 & 45 & 0,0114\\
\hline 3968 & 3886 & 127 & 38 & 45 & 0,0114 & 127 & 37 & 46 & 0,0116\\
\hline 3968 & 3885 & 128 & 38 & 46 & 0,0116 & 128 & 38 & 46 & 0,0116\\
\hline 3968 & 3884 & 129 & 40 & 45 & 0,0114 & 129 & 39 & 46 & 0,0116\\
\hline 3968 & 3883 & 130 & 40 & 46 & 0,0116 & 130 & 40 & 46 & 0,0116\\
\hline 3968 & 3882 & 131 & 41 & 46 & 0,0116 & 131 & 41 & 46 & 0,0116\\
\hline 3968 & 3881 & 132 & 44 & 44 & 0,0111 & 132 & 42 & 46 & 0,0116\\
\hline 3968 & 3880 & 133 & 44 & 45 & 0,0114 & 133 & 46 & 43 & 0,0109\\
\hline 3968 & 3879 & 134 & 44 & 46 & 0,0116 & 134 & 48 & 42 & 0,0106\\
\hline 3968 & 3878 & 135 & 46 & 45 & 0,0114 & 135 & 48 & 43 & 0,0109\\
\hline 3968 & 3877 & 136 & 46 & 46 & 0,0116 & 136 & 50 & 42 & 0,0106\\
\hline 3968 & 3876 & 137 & 48 & 45 & 0,0114 & 137 & 50 & 43 & 0,0109\\
\hline 3968 & 3875 & 138 & 48 & 46 & 0,0116 & 138 & 50 & 44 & 0,0111\\
\hline 3968 & 3874 & 139 & 49 & 46 & 0,0116 & 139 & 50 & 45 & 0,0114\\
\hline 3968 & 3873 & 140 & 52 & 44 & 0,0111 & 140 & 50 & 46 & 0,0116\\
\hline 3968 & 3872 & 141 & 52 & 45 & 0,0114 & 141 & 51 & 46 & 0,0116\\
\hline 3968 & 3871 & 142 & 52 & 46 & 0,0116 & 142 & 52 & 46 & 0,0116\\
\hline 3968 & 3870 & 143 & 54 & 45 & 0,0114 & 143 & 53 & 46 & 0,0116\\
\hline 3968 & 3869 & 144 & 54 & 46 & 0,0116 & 144 & 56 & 44 & 0,0111\\

\hline 3968 & 3868 & 145 & 55 & 46 & 0,0116 & 145 & 57 & 44 & 0,0111\\
\hline 3968 & 3867 & 146 & 56 & 46 & 0,0116 & 146 & 58 & 44 & 0,0111\\
\hline 3968 & 3866 & 147 & 57 & 46 & 0,0116 & 147 & 58 & 45 & 0,0114\\
\hline 3968 & 3865 & 148 & 60 & 44 & 0,0111 & 148 & 58 & 46 & 0,0116\\
\hline 3968 & 3864 & 149 & 60 & 45 & 0,0114 & 149 & 59 & 46 & 0,0116\\
\hline 3968 & 3863 & 150 & 60 & 46 & 0,0116 & 150 & 60 & 46 & 0,0116\\
\hline 3968 & 3862 & 151 & 62 & 45 & 0,0114 & 151 & 61 & 46 & 0,0116\\
\hline 3968 & 3861 & 152 & 62 & 46 & 0,0116 & 152 & 62 & 46 & 0,0116\\
\hline 3968 & 3860 & 153 & 63 & 46 & 0,0116 & 153 & 63 & 46 & 0,0116\\
\hline 3968 & 3859 & 154 & 64 & 46 & 0,0116 & 154 & 64 & 46 & 0,0116\\
\hline 3968 & 3858 & 155 & 65 & 46 & 0,0116 & 155 & 66 & 45 & 0,0114\\
\hline 3968 & 3857 & 156 & 66 & 46 & 0,0116 & 156 & 66 & 46 & 0,0116\\
\hline 3968 & 3856 & 157 & 68 & 45 & 0,0114 & 157 & 67 & 46 & 0,0116\\
\hline 3968 & 3855 & 158 & 68 & 46 & 0,0116 & 158 & 68 & 46 & 0,0116\\
\hline 3968 & 3854 & 159 & 70 & 45 & 0,0114 & 159 & 69 & 46 & 0,0116\\
\hline 3968 & 3853 & 160 & 70 & 46 & 0,0116 & 160 & 70 & 46 & 0,0116\\
\hline 3968 & 3852 & 161 & 71 & 46 & 0,0116 & 161 & 71 & 46 & 0,0116\\
\hline 3968 & 3851 & 162 & 72 & 46 & 0,0116 & 162 & 72 & 46 & 0,0116\\
\hline 3968 & 3850 & 163 & 73 & 46 & 0,0116 & 163 & 73 & 46 & 0,0116\\
\hline 3968 & 3849 & 164 & 74 & 46 & 0,0116 & 164 & 74 & 46 & 0,0116\\
\hline 3968 & 3848 & 165 & 76 & 45 & 0,0114 & 165 & 75 & 46 & 0,0116\\

\hline 3968 & 3847 & 166 & 76 & 46 & 0,0116 & 166 & 76 & 46 & 0,0116\\
\hline 3968 & 3846 & 167 & 77 & 46 & 0,0116 & 167 & 77 & 46 & 0,0116\\
\hline 3968 & 3845 & 168 & 78 & 46 & 0,0116 & 168 & 78 & 46 & 0,0116\\
\hline 3968 & 3844 & 169 & 79 & 46 & 0,0116 & 169 & 79 & 46 & 0,0116\\
\hline 3968 & 3843 & 170 & 80 & 46 & 0,0116 & 170 & 80 & 46 & 0,0116\\
\hline 3968 & 3842 & 171 & 81 & 46 & 0,0116 & 171 & 81 & 46 & 0,0116\\
\hline 3968 & 3841 & 172 & 82 & 46 & 0,0116 & 172 & 82 & 46 & 0,0116\\
\hline 3968 & 3840 & 173 & 84 & 45 & 0,0114 & 173 & 83 & 46 & 0,0116\\
\hline 3968 & 3839 & 174 & 84 & 46 & 0,0116 & 174 & 84 & 46 & 0,0116\\
\hline 3968 & 3838 & 175 & 85 & 46 & 0,0116 & 175 & 85 & 46 & 0,0116\\
\hline 3968 & 3837 & 176 & 86 & 46 & 0,0116 & 176 & 86 & 46 & 0,0116\\
\hline 3968 & 3836 & 177 & 87 & 46 & 0,0116 & 177 & 87 & 46 & 0,0116\\
\hline 3968 & 3835 & 178 & 88 & 46 & 0,0116 & 178 & 88 & 46 & 0,0116\\
\hline 3968 & 3834 & 179 & 89 & 46 & 0,0116 & 179 & 89 & 46 & 0,0116\\
\hline 3968 & 3833 & 180 & 90 & 46 & 0,0116 & 180 & 90 & 46 & 0,0116\\
\hline 3968 & 3832 & 181 & 92 & 45 & 0,0114 & 181 & 91 & 46 & 0,0116\\
\hline 3968 & 3831 & 182 & 92 & 46 & 0,0116 & 182 & 92 & 46 & 0,0116\\

\hline 3968 & $3968-\ell_{\infty}$ & $\rho_{\ell_\infty} \geq 183$ & $\ell_{\infty}-45$ & 46 & 0,0116 & $\rho_{\ell_0} \geq 183$ & $\ell_{0}-45$ & 46 & 0,0116 \\

\end{longtable}
}
\end{center}

Table \ref{TabMeglio} provides some examples in which codes of type $C_{\ell_0}(P_0)$ have better parameters than codes of type $C_{\ell_\infty}(P_\infty)$.
In particular, the length $n$ of the two codes is $3968$, the dimension $k_0$ and the Feng-Rao designed minimum distance $d_{ORD}^0$ of $C_{\ell_0}(P_0)$ are greater than or equal to the corresponding parameters $k_\infty$ and $d_{ORD}^\infty$ of $C_{\ell_\infty}(P_\infty)$, and the designed Singleton defect $\delta_0=n+1-k_0-d_{ORD}^0$ of $C_{\ell_0}(P_0)$ is strictly smaller than the designed Singleton defect $\delta_\infty=n+1-k_\infty-d_{ORD}^{\infty}$ of $C_{\ell_\infty}(P_\infty)$.

\begin{table}[htbp]
\begin{scriptsize}
\caption{Designed Singleton defect of $C_{\ell_0}(P_0)$ and $C_{\ell_\infty}(P_\infty)$, $q^n=2^5$}\label{TabMeglio}
\begin{center}
\begin{tabular}{|c|c|c|c|c|c|c|c|c|c|c|c|c|c|c|c|}
\hline $\ell_0$ & 3 & 4 & 5 & 6 & 8 & 9 & 10 & 19 & 20 & 21 & 22 & 23 & 24 & 26 & 27\\

\hline $\ell_\infty$ & 4 & 5 & 6 & 7 & 9 & 10 & 11 & 20 & 21 & 22 & 23 & 24 & 25 & 27 & 28\\

\hline $\delta_\infty-\delta_0$ & 1 & 1 & 1 & 1 & 1 & 1 & 1 & 1 & 1 & 1 & 1 & 1 & 1 & 1 & 1\\

\hline\hline $\ell_0$ & 28 & 29 & 30  & 31  & 32 & 34 & 35 & 36 & 37 & 38 & 39 & 40 & 41 & 42 & 43\\

\hline $\ell_\infty$ & 29 & 30 & 31  & 32  & 33 & 35 & 36 & 37 & 38 & 39 & 40 & 41 & 42 & 43 & 44\\

\hline $\delta_\infty-\delta_0$ & 1 & 1 & 1 & 1  & 1 & 1 & 1 & 1 & 1 & 1 & 1 & 1 & 1 & 1 & 1\\

\hline\hline $\ell_0$ & 44 & 45 & 46 & 47 & 48 & 49 & 50 & 51  & 52 & 55 & 56 & 56 & 57 & 57 & 58\\

\hline $\ell_\infty$ & 45 & 46 & 47 & 48 & 49 & 50 & 51 & 52  & 53 & 56 & 56 & 57 & 57 & 58 & 58\\

\hline $\delta_\infty-\delta_0$ & 1 & 1 & 1 & 1 & 1 & 1 & 1 & 1 & 1 & 1 & 5 & 6 & 6 & 7 & 6\\

\hline\hline $\ell_0$ & 58 & 59 & 59 & 60 & 60 & 61 & 61 & 62 & 63 & 64 & 65 & 66  & 66 & 67 & 67\\

\hline $\ell_\infty$ & 59 & 59 & 60 & 60 & 61 & 61 & 62 & 63 & 64 & 65 & 66 & 66  & 67 & 67 & 68\\

\hline $\delta_\infty-\delta_0$ & 7 & 6 & 7 & 6 & 7 & 6 & 1 & 1 & 1 & 1 & 1 & 4  & 6 & 7 & 6\\

\hline\hline $\ell_0$ & 68 & 68 & 69 & 77 & 77 & 78 & 88 & 88 & 89 & 89 & 90 & 90 & 91 & 91 &\\

\hline $\ell_\infty$ & 68 & 69 & 69 & 77 & 78 & 78 & 88 & 89 & 89 & 90 & 90 & 91 & 91 & 92 &\\

\hline $\delta_\infty-\delta_0$ & 5 & 4 & 5 & 4 & 4 & 4 & 2 & 3 & 4 & 3 & 2 & 3 & 4 & 3 &\\

\hline\hline $\ell_0$ & 92 & 92  & 93 & 93 & 94 & 99 & 99 & 100 & 100 & 101 & 101 & 102 & 110 & 110 & \\

\hline $\ell_\infty$ & 92 & 93  & 93 & 94 & 94 & 99 & 100 & 100 & 101 & 101 & 102 & 102 & 110 & 111 &  \\

\hline $\delta_\infty-\delta_0$ & 2 & 1  & 2 & 2 & 1 & 2 & 2 & 2 & 2 & 2 & 2 & 1 & 1 & 1 &  \\
\hline
\end{tabular}
\end{center}
\end{scriptsize}
\end{table}

\section{Quantum codes from one-point AG codes on the GGS curves}\label{Sec:Application2}

In this section we use families of one-point AG codes from the GGS curve to construct quantum codes. The main ingredient is the so called {\it CSS contruction} which enables to construct quantum codes from classical linear codes; see \cite[Lemma 2.5]{LGP}.

We denote by $q$ a prime power.
A $q$-ary quantum code $Q$ of length $N$ and dimension $k$ is defined to be a $q^k$-dimensional Hilbert subspace of a $q^N$-dimensional Hilbert space $\mathbb H=(\mathbb C^q)^{\otimes n}=\mathbb C^q\otimes\cdots\otimes\mathbb C^q$. If $Q$ has minimum distance $D$, then $Q$ can correct up to $\lfloor\frac{D-1}{2}\rfloor$ quantum errors.
The notation $[[N,k,D]]_q$ is used to denote such a quantum code $Q$.
For a $[[N,k,D]]_q$-quantum code the quantum Singleton bound holds, that is, the minimum distance satisfies $D\leq 1+(N-k)/2$.
The quantum Singleton defect is $\delta^Q:=N-k-2D+2\geq0$, and the relative quantum Singleton defect is $\Delta^Q:=\delta^Q/N$.
If $\delta^Q=0$, then the code is said to be quantum MDS.
For a detailed introduction on quantum codes see \cite{LGP} and the references therein.

\begin{lemma} \label{ccs} {\rm (CSS construction)}
Let $C_1$ and $C_2$ denote two linear codes with parameters $[N,k_i,d_i]_q$, $i=1,2$, and assume that $C_1 \subset C_2$. Then there exists an $[[N,k_2-k_1,D]]_q$ code with $D=\min\{wt(c) \mid c \in (C_2 \setminus C_1) \cup (C_1^\perp \setminus C_2^\perp)\}$, where $wt(c)$ is the weight of $c$.
\end{lemma}

We consider the following \textit{general t-point construction} due to La Guardia and Pereira; see \cite[Theorem 3.1]{LGP}. It is a direct application of Lemma \ref{ccs} to AG codes.

\begin{lemma} \label{lem1} {\rm (General t-point construction)}
Let $\mathcal X$ be a nonsingular curve over $\mathbb F_q$ with genus $g$ and $N+t$ distinct $\mathbb F_q$-rational points, for some $N,t>0$.
Assume that $a_i,b_i$, $i=1,\ldots,t$, are positive integers such that $a_i \leq b_i$ for all $i$ and $2g-2 < \sum_{i=1}^{t} a_i < \sum_{i=1}^t b_i < N$. Then there exists a quantum code with parameters $[[N,k,D]]_{q}$ with $k=\sum_{i=1}^{t} b_i - \sum_{i=1}^{t} a_i$ and $D \geq \min \big\{ N - \sum_{i=1}^{t} b_i, \sum_{i=1}^{t} a_i - (2g-2)\big\}$.
\end{lemma}

Let $n\geq 5$ be an odd integer.
 We apply Lemma \ref{lem1} to one-point codes on the GGS curve.

\begin{proposition} \label{qua}
Let $a,b\in\mathbb N$ be such that
$$(q-1)(q^{n+1}+q^n-q^2)-2<a<b<q^{2n+2}-q^{n+3}+q^{n+2}.$$
Then there exists a quantum code with parameters $[[N,b-a,D]]_{q^{2n}}$, where 
$$N=q^{2n+2}-q^{n+3}+q^{n+2},$$
$$D \geq \min \left \{ q^{2n+2}-q^{n+3}+q^{n+2}-b, a- (q-1)(q^{n+1}+q^n-q^2)+2 \right \}.$$
\end{proposition}

\begin{proof}
Let $GGS(q,n)$ be the GGS curve with equations \eqref{GGS_equation}, genus $g$, and infinite point $P_\infty$.
Consider the divisors $\overline D$ as in \eqref{Dbarra}, $G_1=a P_\infty$, and $G_2=b P_\infty$.
Note that ${\rm supp} (G_1) \cap {\rm supp}(\overline D)={\rm supp}(G_2) \cap {\rm supp}(\overline D)=\emptyset$.
From Lemma \ref{lem1}, there exists a quantum code with parameters $[[N,b-a,D]]_{q^{2n}}$, where $D \geq \min \big \{N-b, a-(2g-2) \big \}=\min \big \{ q^{2n+2}-q^{n+3}+q^{n+2}-b, a- (q-1)(q^{n+1}+q^n-q^2)-2 \big \}$.
\end{proof}

Another application of the CSS construction can be obtained looking at the dual codes of the one-point codes from the GGS curve. Let $P \in GGS(q,n)$. Fix $a=\rho_\ell \in H(P)$ and $b=\rho_{\ell+s} \in H(P)$ with $C_2=C_\ell(P)=C_\ell$ and $C_1=C_{\ell+s}(P)=C_{\ell+s}$, where $s \geq 1$. Clearly $C_1 \subset C_2$, as $C_{\ell} \subsetneq C_{\ell+s}$ for every $s \geq 1$. The dimensions of $C_2$ and $C_1$ are $k_2=N-h_\ell$ and $k_1=N-h_{\ell+s}=N-h_\ell-s$ respectively, where $h_i$ denotes the number of non-gaps at $P$ which are smaller than or equal to $i$. Thus, $k_2-k_1=s$. According to the CSS construction, these choices induce an $[[N, s,D]]_{q^{2n}}$ quantum code, where $N=q^{2n+2}-q^{n+3}+q^{n+2}$ and $D=\min\{wt(c) \mid c \in (C_2 \setminus C_1) \cup (C_1^\perp \setminus C_2^\perp)\}=\min\{wt(c) \mid c \in (C_{\ell} \setminus C_{\ell+s}) \cup (C(D,G_1) \setminus C(D,G_2)) \},$ with $G_2=\rho_\ell P$ and $G_1=\rho_{\ell+s} P$. 
In particular,
\begin{equation} \label{Dquant}
D \geq \min \{d_{ORD}(C_\ell), d_1\},
\end{equation}
where $d_1$ denotes the minimum distance of the code $C(D,G_1)$. Following this construction and using an improvement of Inequality \eqref{Dquant}, the next theorem is obtained. 

\begin{theorem} \label{quant1}
Let $g=(q-1)(q^{n+1}+q^n-q^2)/2$ and $N=q^{2n+2}-q^{n+3}+q^{n+2}$.
For every $\ell\in\left[3g-1,N-g\right]$ and $s\in\left[1,N-2\ell\right]$, there exists a quantum code with parameters $[[N,s,D]]_{q^{2n}}$, where $D \geq \ell+1-g$.
\end{theorem}

\begin{proof}
Since $\ell\geq3g-1$, we have $\rho_{\ell+s}=g-1+\ell+s$, and hence $d_1\geq N-\deg(G_1)=N-\rho_{\ell+s}=N-\ell-s-g+1$.
From Theorem \ref{fengrao}, $d_{ORD}(C_\ell)=\ell+1-g$. Thus, $D \geq \min \{d_{ORD}(C_\ell), d_1\}=\ell+1-g$. The claim follows.
\end{proof}

For fixed $q$, we can construct as a direct consequence of Theorem \ref{quant1} families of quantum codes depending on $n$ such that their relative quantum Singleton defect goes to zero as $n$ goes to infinity. 
An example is the following.

\begin{corollary} \label{quant2}
Let $g=(q-1)(q^{n+1}+q^n-q^2)/2$ and $N=q^{2n+2}-q^{n+3}+q^{n+2}$.
For every $\ell\in[3g-1,N-g]$, fix $s=N-2\ell$.
Then there exists a quantum code with parameters $[[N,s,D]]_{q^{2n}}$ with $D\geq\ell+1-g$, whose relative quantum Singleton defect $\Delta_n^Q=(N-s-2D+2)/N$ satisfies
$$\Delta_n^Q\leq\frac{2g}{N}=\frac{(q-1)(q^{n+1}+q^n-q^2)}{q^{2n+2}-q^{n+3}+q^{n+2}}.$$
Hence, $\lim_{n \rightarrow \infty}{\Delta_n^Q} = 0$.
\end{corollary}


Using the computation of $d_{ORD}(C_{\ell}(P_\infty))$ in Section \ref{Sec:d_ORD}, we produce infinite families of quantum codes in which the lower bound in \eqref{Dquant} is explicitely determined. We look at those cases for which \eqref{Dquant} reads $D \geq d_{ORD}(C_{\ell}(P_\infty))>\ell+1-g$ and this bound is better than the one stated in Theorem \ref{quant1}. According to Proposition \ref{farran}, we choose $\rho_\ell \in H(P_\infty)$ such that $\rho_{\ell+1}=(0,1,k)$ for some $k \in [m,2m)$.
\begin{proposition}
Let $q=2^n$ for $n \geq 5$ odd, $g=(q-1)(q^{n+1}+q^n-q^2)/2$, and $N=q^{2n+2}-q^{n+3}+q^{n+2}$.
Let $\ell\in[g,3g-1]$ be such that $\rho_{\ell+1}\in H(P_\infty)$ is of type $(0,1,k)$ for some $k \in [m,2m)$. Let $s\in[1,N-2\ell-5]$.
Then there exists a quantum code with parameters $[[N,s,D]]_{q^{2n}}$ where
$$D \geq \ell+1-g+\begin{cases} 5, & \textrm{if} \ k<m \ \textrm{or} \ m \leq k <\frac{9m-11}{8}, \\ 3, & \textrm{if} \ \frac{9m-11}{8} \leq k < \frac{11m-9}{8}, \\ 1, & \textrm{if} \ \frac{11m-9}{8} \leq k. \end{cases}$$
 \end{proposition} 
 \begin{proof} Arguing as in the proof of Theorem \ref{quant1}, we have that $d_1 \geq N-\ell-s-g+1$. Thus, from Proposition \ref{perquant} and Lemma \ref{dord6}, Inequality \eqref{Dquant} reads $$D \geq d_{ORD}(C_{\ell}(P_\infty))=\begin{cases} 8k-7m+13, & \textrm{if} \ k<m \ \textrm{or} \ m \leq k <\frac{9m-11}{8}, \\ 8k-7m+11, & \textrm{if} \ \frac{9m-11}{8} \leq k < \frac{11m-9}{8}, \\ 8k-7m+9, & \textrm{if} \ \frac{11m-9}{8} \leq k. \end{cases}$$ Since $\ell+1-g=\rho_{\ell+1}-2g+1=2m+8k-(9m-7)+1=8k-7m+8$, the claim follows. \end{proof}

\section{Convolutional codes from one-point AG codes on the GGS curves}\label{Sec:Application3}

In this section we use a result due to De Assis, La Guardia, and Pereira \cite{ALGR} which allows to construct unit-memory convolutional codes with certain parameters $(N,k,\gamma;m,d_f)_q$ starting from AG codes.

Consider the polynomial ring $R=\fq[X]$. A convolutional code $C$ is an $R$-submodule of rank $k$ of the module $R^N$.
Let $G(X)=(g_{ij}(X))\in\fq[X]^{k\times N}$ be a generator matrix of $C$ over $\fq[X]$, $\gamma_i=\max\{\deg g_{ij}(X)\mid1\leq j\leq N\}$, $\gamma=\sum_{i=1}^k \gamma_i$, $m=\max\{\gamma_i\mid1\leq i\leq k\}$, and $d_f$ be the minimum weight of a word $c\in C$.
Then we say that $C$ has length $N$, dimension $k$, degree $\gamma$, memory $m$, and free distance.
If $m=1$, $C$ is said to be a unit-memory convolutional code.
In this case we use for  $C$ the notation $(N,k,\gamma;m,d_f)_q$.
For a detailed introduction on convolutional codes see \cite{ALGR,RS1999} and the references therein.

\begin{lemma} \label{lem21}{\rm (\!\!\cite[Theorem 3]{ALGR})}
Let $\mathcal X$ be a nonsingular curve over $\mathbb F_q$ with genus $g$. Consider an AG code $C^{\bot}(D,G)$ with $2g-2<\deg(G)<N$. Then there exists a unit-memory convolutional code with parameters $(N,k-\ell,\ell;1,d_f \geq d)_q$, where $\ell \leq k/2$, $k=\deg(G)+1-g$ and $d \geq N-\deg(G)$.
\end{lemma}

We apply Lemma \ref{lem21} to one-point AG codes from the GGS curve.
\begin{proposition}\label{finiamo}
Consider the $\mathbb F_{q^{2n}}$-maximal GGS curve $GGS(q,n)$ and let $\rho_\ell \in H(P_\infty)$ be such that $(q-1)(q^{n+1}+q^n-q^2)-2<\rho_\ell<N$, where $N=q^{2n+2}-q^{n+3}+q^{n+2}$. Then there exists a unit-memory convolutional code with parameters $(N, k- s,s;1, d_f \geq d_{ORD}(C_\ell(P_\infty)))$, where $k=\rho_{\ell}+1-\frac{(q-1)(q^{n+1}+q^n-q^2)}{2}$ and $s \leq k/2$.
\end{proposition}

\begin{proof}
The result follows from Lemma \ref{lem21}. The inequality $d_f \geq d_{ORD}(C_\ell(P_\infty))$ follows from $d_f \geq d$ and Theorem \ref{fengrao} applied to the dual code $C_\ell(P_\infty)$.
\end{proof}

In particular, Theorem \ref{fengrao} yields the following corollary.

\begin{corollary}
Consider the $\mathbb F_{q^{2n}}$-maximal GGS curve $GGS(q,n)$ and let $\rho_\ell \in H(P_\infty)$ be such that $(q-1)(q^{n+1}+q^n-q^2)-2<\rho_\ell<N$, where  $N=q^{2n+2}-q^{n+3}+q^{n+2}$ and $\ell \geq 3\frac{(q-1)(q^{n+1}+q^n-q^2)}{2}$. Then there exists a unit-memory convolutional code with parameters $(N, k- s,s;1, d_f)$, where $k=\rho_{\ell}+1-\frac{(q-1)(q^{n+1}+q^n-q^2)}{2}$, $s \leq k/2$, and $d_f \geq \ell+1-\frac{(q-1)(q^{n+1}+q^n-q^2)}{2}$.
\end{corollary}

\section{The Automorphism group of $C(\overline D,\ell P_{\infty})$}\label{Sec:Auto}

In this section we investigate the automorphism group of the code $C(\overline D,\ell P_{\infty})$, where $\overline D$ is as in \eqref{Dbarra}.

\begin{lemma}\label{Orbite}
The automorphism group $\aut(GGS(q,n))$ has exactly two short orbits on $GGS(q,n)$; one consists of $P_\infty$, the other consists of the $q^3$ $\mathbb{F}_{q^2}$-rational points other than $P_\infty$.
\end{lemma}

\begin{proof}
From \cite{GMP,GOS}, $\aut(GGS(q,n))=Q\rtimes\Sigma$, where $Q=\{Q_{a,b}\mid a,b\in\mathbb{F}_{q^2},a^q+a=b^{q+1}\}$ and $\Sigma=\langle g_{\zeta} \rangle$, with
\begin{equation}\label{AutDef} Q_{a,b}= \begin{pmatrix} 1 & b^q & 0 & a \\ 0 & 1 & 0 & b \\ 0 & 0 & 1 & 0 \\ 0 & 0 & 0 & 1 \end{pmatrix},\quad
g_{\zeta} = \begin{pmatrix} \zeta^{q^n+1}&0&0&0 \\ 0&\zeta^{\frac{q^n+1}{q+1}}&0&0 \\ 0&0&\zeta&0 \\ 0&0&0&1 \end{pmatrix}, \end{equation}
$\zeta$ a primitive $(q^n+1)(q-1)$-th root of unity.
Therefore, $\aut(GGS(q,n))$ fixes $P_\infty$. Also, $\aut(GGS(q,n))$ acts transitively on the $q^3$ affine points of $GGS(q,n)$ having zero $Z$-coordinate, which coincide with the $\mathbb F_{q^2}$-rational points of $GGS(q,n)$ other than $P_\infty$.

Suppose $\aut(GGS(q,n))$ has another short orbit $\mathcal O$.
Since $GGS(q,n)$ has zero $p$-rank and $\aut(GGS(q,n))$ fixes $P_\infty$, $\mathcal O$ is tame. Hence, by Schur-Zassenhaus Theorem \cite[Theorem 9.19]{Rose}, the stabilizer of a point $P\in\mathcal O$ is contained up to conjugation in $\Sigma$.
This is a contradiction, as $\Sigma$ acts semiregularly out of the plane $Z=0$.
\end{proof}

Note from \eqref{AutDef} that $\aut(GGS(q,n))$ is defined over $\mathbb F_{q^{2n}}$.
Let $\pi_a$ be the plane $Z=a$. The points of $\pi_0\cap GGS(q,n)$ are exactly the $q^3+1$ $\fqq$-rational points of $GGS(q,n)$, while all coordinates of any point of $GGS(q,n)\setminus\pi_0$ are not in $\fqq$.
The group $\Sigma$ fixes all points in $\pi_0\cap GGS(q,n)$ and acts semiregularly on the planes $\pi_a$, 
while the group Q acts transitively on $\pi_0\cap GGS(q,n)$ and fixes $GGS(q,n)\cap\pi_a$ for all $a$.
Also, $Q$ acts faithfully on the Hermitian curve $\mathcal{H}_q:Y^{q+1}=X^q+X$ by $(X,Y,T)\mapsto\bar Q\cdot(X,Y,T)$, where $\bar Q$ is obtained from $Q$ deleting the third row and column.

\begin{proposition}
The automorphism group of ${C}(\overline D,\ell P_{\infty})$ contains a subgroup isomorphic to
$$ (\aut(GGS(q,n))\rtimes\aut(\mathbb{F}_{q^{2n}}))\rtimes\mathbb{F}_{q^{2n}}^{*}. $$
\end{proposition}

\begin{proof}
The set $S_{\sigma}$ of points of $GGS(q,n)$ fixed by a non-trivial automorphism $\sigma$ of $\aut_{\fqnq}(GGS(q,n))=\aut(GGS(q,n))$ has size $N_{\sigma}\leq q^3+1$.
In fact, if $\sigma\notin Q$, then $S_{\sigma}\subseteq\pi_0$.
If $\sigma\in Q$, then from $\sigma(P_\infty)=P_\infty$ we have that the induced automorphism $\bar\sigma\in\aut(\mathcal{H}_q)$ fixes only $\fqq$-rational points of $\mathcal H_q$; hence, $\sigma$ fixes only $\fqq$-rational points of $GGS(q,n)$, that is, $S_{\sigma}\subseteq\pi_0$.
Since $|GGS(q,n)\cap\pi_0|=q^3+1$, $N_{\sigma}\leq q^3+1$.
Now the claim follows from \cite[Proposition 2.3]{BMZ}.
\end{proof}

\begin{proposition}
If $q^n+1 \leq \ell \leq q^{n+2}-q^3$ and $\{\ell,\ell-1\}\subset H(P_\infty)$, then
$$ \aut({C}(\overline D,\ell P_{\infty})) \cong (\aut(GGS(q,n))\rtimes\aut(\mathbb{F}_{q^{2n}}))\rtimes\mathbb{F}_{q^{2n}}^{*}. $$
\end{proposition}

\begin{proof}
We apply \cite[Theorem 3.4]{GK2008}.
\begin{itemize}
\item The divisor $G=\ell P_\infty$ is effective.
\item A plane model of degree $q^n+1$ for $GGS(q,n)$ is
\begin{equation}\label{modello} \Pi(GGS(q,n)):\quad Z^{q^n+1} = X^{q^3}+X - (X^q+X)^{q^2-q+1}. \end{equation}
In fact, $Z^{m(q+1)}=Y^{q+1}h(X)^{q+1}=X^{q^3}+X - (X^q+X)^{q^2-q+1}$; also, Equation \eqref{modello} is irreducible since it defines a Kummer extension $\mathbb K(x,z)/\mathbb K(x)$ totally ramified over the pole of $x$.
Therefore, $\mathbb K(GGS(q,n))=\mathbb K(x,z)$, and $x,z\in\cL(G)$ from the assumption $\ell\geq q^n+1$. 
\item The support of $D$ is preserved by the Frobenius morphism $\varphi:(x,z)\mapsto(x^p,z^p)$, since $\varphi(P_\infty)=P_\infty$ and ${\rm supp}(D)=GGS(q,n)(\mathbb F_{q^{2n}})\setminus\{P_\infty\}$.
\item Let $N$ be the length of ${C}(\overline D,\ell P_{\infty})$. Then the condition $N>\deg(G)\cdot\deg(\Pi(GGS(q,n)))$ reads
$$ q^{2n+2}-q^{n+3}+q^{n+2} > \ell (q^n+1), $$
which is implied by the assumption $\ell\leq q^{n+2}-q^3$.
\item
\begin{itemize}
\item If $P=P_\infty$, then $\cL(G)\ne\cL(G-P)$ since $\ell\in H(P_\infty)$.
\item If $P\ne P_\infty$, then $1\in\cL(G)\setminus\cL(G-P)$.
\item If $P=Q=P_\infty$, then $\cL(G-P)\ne\cL(G-P-Q)$ since $\ell-1\in H(P_\infty)$.
\item If $P=P_\infty$ and $Q\ne P_\infty$, then $1\in\cL(G-P)\setminus\cL(G-P-Q)$.
\item If $P\ne P_\infty$ and $Q=P_\infty$, then $f-\mu\in\cL(G-P)\setminus\cL(G-P-Q)$, where $f\in\cL(G)$ has pole divisor $\ell P_\infty$ and $\mu=f(P)$.
\item If $P,Q\ne P_\infty$ and $P\ne Q$, choose $f=z-z(P)$ or $f=x-x(P)$ according to $z(P)\ne z(Q)$ or $x(P)\ne x(Q)$; then $f\in\cL(G-P)\setminus\cL(G-P-Q)$.
\item If $P=Q\ne P_\infty$, then $z-z(P)\in\cL(G-P)\setminus\cL(G-P-Q)$.
\end{itemize}
\end{itemize}
Thus we can apply \cite[Theorem 3.4]{GK2008} to prove the claim.
\end{proof}

\end{document}